\documentclass[11pt]{amsart}

\usepackage{amsmath}
\usepackage{amssymb}
\usepackage{url}
\usepackage{amscd}

\usepackage{bbm}
\usepackage{dsfont}
\usepackage{yfonts}
\usepackage{mathrsfs}
\usepackage{mathtools}
\mathtoolsset{showonlyrefs,showmanualtags}

\usepackage[backref=page]{hyperref} 
\hypersetup{
    colorlinks=true,       
    linkcolor=blue,        
    citecolor=magenta,       
    filecolor=magenta,    
    urlcolor=cyan          
}

 \newtheorem{theorem}{Theorem}[section]

 \newtheorem{corollary}[theorem]{Corollary}
 \newtheorem{lemma}[theorem]{Lemma}
 \newtheorem{proposition}[theorem]{Proposition}

 \theoremstyle{definition}
 
 \theoremstyle{remark}
 \newtheorem{remark}[theorem]{Remark}
 \newtheorem*{remark_v2}{Remark}
 
 \numberwithin{equation}{section}

\def\C{\mathbb C}

\def\D{{\mathbb D}}

\def\i{\textbf{i}}

\def\I{{1\!\!1}}

\begin{document}

\title[Volterra-type inner derivations]{Volterra-type inner derivations\\ on Hardy spaces}

\author[H. Arroussi]{H. Arroussi}
\address{\newline Hicham Arroussi \newline Department of Mathematics and Statistics, University of Helsinki, Finland
\newline Department of Mathematics and Statistics, University of Reading, England}

\email{arroussihicham@yahoo.fr, h.arroussi@reading.ac.uk}

\author[C. Tong]{C. Tong}
\address{\newline Cezhong Tong\newline Department of Mathematics,
	Hebei University of Technology, Tianjin 300401, P.R. China.}
\email{ctong@hebut.edu.cn,
	cezhongtong@hotmail.com}

\author[J. A. Virtanen]{J. A. Virtanen}

\address{\newline Jani A. Virtanen\newline Department of Physics and Mathematics, University of Eastern Finland\newline Department of Mathematics and Statistics, University of Reading, England\newline Department of Mathematics and Statistics, University of Helsinki, Finland}
\email{jani.virtanen@uef.fi, j.a.virtanen@reading.ac.uk, \newline jani.virtanen@helsinki.fi}

\author[Z. Yuan]{Z. Yuan}

\address{\newline Zixing Yuan\newline School of Mathematics and Statistics,
Wuhan University,  Wuhan 430072, P.R. China.}

\email{zxyuan.math@whu.edu.cn}

\begin{abstract}
A classical result of Calkin [Ann.~of Math.~(2) 42 (1941), pp.~839--873] says that an inner derivation $S\mapsto [T,S] = TS-ST$ maps the algebra of bounded operators on a Hilbert space into the ideal of compact operators if and only if $T$ is a compact perturbation of the multiplication by a scalar. In general, an analogous statement fails for operators on Banach spaces. To complement Calkin's result, we characterize Volterra-type inner derivations on Hardy spaces using generalized area operators and compact intertwining relations for Volterra and composition operators. Further, we characterize the compact intertwining relations for multiplication and composition operators between Hardy and Bergman spaces.
\end{abstract}


\maketitle

\section{Introduction}
Let $\mathscr A$ be a Banach algebra over the complex field. A linear map $D: \mathscr A\to \mathscr A$ is a derivation if $D(xy)=xD(y)+D(x)y$ for all $x,y\in\mathscr A$. Over the last half century, there have been plenty of results giving conditions on a derivation of a Banach algebra implying that its range is contained in some ideal. One of the most famous results given by Singer and Wermer \cite[Theorem 1]{SW} says that every continuous derivation of a commutative Banach algebra maps into the Jacobson radical of the algebra. Previously Calkin \cite[Theorem 2.9]{Ca} proved that an \emph{inner derivation} $X\mapsto [T,X]:=TX-XT$ maps the algebra of all bounded operators on a Hilbert space to the ideal of all compact operators if and only if $T$ is a compact perturbation of a scalar operator. Notice that this conclusion fails to hold true on the Banach spaces in general (see \cite[p.~288]{ST}). In this paper, we are interested in Volterra-type inner derivations on Hardy spaces, and, in particular, give characterizations which complement and in a sense extend some aspects of Calkin's work to the algebras of bounded linear operators on Hardy spaces.

To state our main results, we recall some basic definitions. Let $H(\mathbb{D})$ denote the class of all analytic functions in the unit disk $\mathbb{D}$ of the complex plane $\C$ and let $S(\mathbb{D})$ be the collection of all analytic self-maps of $\mathbb{D}.$ 

For $0<p<\infty$, the Hardy space $\mathcal{H}^{p}$ is defined to be the Banach space of all analytic functions $f$ in $\mathbb{D}$ with
$$
\|f\|_{\mathcal{H}^{p}}:=\left(\sup _{0<r<1} \frac{1}{2 \pi} \int_{0}^{2 \pi}\left|f\left(r e^{i \theta}\right)\right|^{p} \mathrm{d} \theta\right)^{1 / p}<\infty .
$$

For $0<p<\infty$ and $\alpha>-1$, the weighted Bergman space $\mathcal A_{\alpha}^{p}(\mathbb{D})$ consists of all analytic functions $f$ in $\mathbb{D}$ for which
$$
\|f\|_{\mathcal A_{\alpha}^{p}}:=\left(\int_{\mathbb{D}}|f(z)|^{p} \mathrm{~d} A_{\alpha}(z)\right)^{\frac{1}{p}}<\infty,
$$
where $\mathrm{d} A_{\alpha}(z)=(1+\alpha)\left(1-|z|^{2}\right)^{\alpha} \mathrm{d} A(z)$ and $\mathrm{d}A(z) =\mathrm{d}x \mathrm{d}y/\pi$ is the normalized area measure.

For $a \in \mathbb{D}$, the M\"{o}bius map $\psi_{a}$ of the disk that interchanges $z$ and 0 is defined by
$$
\psi_{a}(z)=\frac{a-z}{1-\bar{a} z}, \quad z \in \mathbb{D}.
$$
It is well known that
\begin{align}\label{b}
1-\left|\psi_{a}(z)\right|^{2}=\frac{\left(1-|a|^{2}\right)\left(1-|z|^{2}\right)}{|1-\bar{a} z|^{2}}.
\end{align}
Let $\mathrm{Aut}(\mathbb{D})$ denote the automorphism group of $\mathbb{D}$. It is well known in elementary complex analysis that every $\psi \in \mathrm{Aut}(\mathbb{D})$ has the form
$$
\psi(z)=e^{i\theta} \psi_{a}(z), \quad \theta \in [0, 2\pi),\ a \in \mathbb{D}.
$$
The space of analytic functions on $\mathbb{D}$ of bounded mean oscillation, denoted by $\mathcal{B M O A}$,        consists of functions $f$ in $\mathcal{H}^{2}$ such that
$$
\|f\|_{{\mathcal{B M O A}}}^2=|f(0)|^2+\sup_{I} \frac{1}{|I|} \int_I\left|f(\theta)-f_I\right|^2 d\theta<+\infty,
$$
where
$$
f_I=\frac{1}{|I|} \int_I f(\theta)\, d\theta
$$
is the length of $I$. The closure in $\mathcal{B M O A}$, of the set of all polynomials is called $\mathcal{V M O A}$. By \cite{Zhu}, we know that $f \in \mathcal{B M O A}$ if and only if
\begin{align}\label{v}
\sup _{a \in \mathbb{D}} \int_{\mathbb{D}}(1-\left|\psi_{a}(z)\right|^{2})| f^{\prime}(z)|^{2} \mathrm{d} A(z)<\infty,
\end{align}
and $f \in \mathcal{V M O A}$ if and only if
\begin{align}\label{c}
\lim _{|a| \rightarrow 1} \int_{\mathbb{D}}(1-\left|\psi_{a}(z)\right|^{2})| f^{\prime}(z)|^{2} \mathrm{d} A(z)=0 .
\end{align}

For $f \in H(\mathbb{D})$, every $\varphi \in S(\mathbb{D})$ induces a composition operator
 $C_{\varphi}$ by $C_{\varphi} f=f \circ \varphi.$ If $\varphi(z)=e^{\i \theta}z$ for $\theta\in[0,2\pi]$, we
 call $C_\varphi$ a rotation composition operator. The rotation composition operators play a crucial role in the proofs of Theorems \ref{567} and \ref{2345}. The boundedness and compactness of composition 
 operators on various analytic function spaces have been studied intensively in the past few decades (see, e.g., \cite{Cowen} and \cite{JS}).

For $g \in H(\mathbb{D})$, the Volterra-type operators $J_{g}$ and $I_g$ are defined by
$$
J_{g} f(z)=\int_{0}^{z} f(\zeta) g^{\prime}(\zeta)\, \mathrm{d} \zeta\quad\mbox{and}\quad I_{g} f(z)=\int_{0}^{z} f^{\prime}(\zeta) g(\zeta)\, \mathrm{d} \zeta
$$
for $z \in \mathbb{D}$ and $f \in H(\mathbb{D})$.  The operators $J_{g}$ and $I_{g}$ are
close companions because of their relations to the multiplication operator $M_{g} f(z)=$ $g(z) f(z)$.
To see this, use integration by parts to obtain
$$
M_{g} f=f(0) g(0)  +J_{g} f+I_{g} f.
$$
The discussion of Volterra-type operators $J_g$ and $I_g$ first
arose in connection with semigroups of composition operators---for further details and background, see \cite{AR} and also \cite{A, AGV} for these types of operators acting on weighted Bergman spaces.

Let $\mathscr B(\mathcal{H}^p)$ be the Banach algebra of bounded linear operators on the Hardy space $\mathcal{H}^p$, where $0<p<\infty.$
The two classes of \emph{Volterra-type inner derivations} $D(J_g)$ and $D(I_g)$ induced by $g\in H(\mathbb D)$ on $\mathscr B(\mathcal{H}^p)$ are defined by
$$
	D(J_{g}): \mathscr B(\mathcal{H}^p)\rightarrow \mathscr B(\mathcal{H}^p),\; T\mapsto [J_{g},T]
$$
(referred to as the $J_g$ inner derivation) and
$$
	D(I_{g}): \mathscr B(\mathcal{H}^p)\rightarrow \mathscr B(\mathcal{H}^p),\; T\mapsto [I_{g},T]
$$
(referred to as the $I_g$ inner derivation).	

\medskip

We can now state our main results.

\begin{theorem}\label{thm-a}
For $0<p<\infty$, the $J_{g}$ inner derivation $D(J_{g})$ on $\mathscr B(\mathcal{H}^p)$ maps into the ideal of compact operators if and only if $g$ belongs to $\mathcal{VMOA}$.
\end{theorem}

\begin{theorem}\label{thm-b}
For $0< p<\infty,$ the $I_{g}$ inner derivation $D(I_{g})$ on $\mathscr B(\mathcal{H}^p)$ maps into the ideal of compact operators if and only if $g$ is a complex scalar.
\end{theorem}

\medskip

The proofs of Theorems \ref{thm-a} and \ref{thm-b} are given in Sections \ref{proof-of-thm-a} and \ref{proof-of-thm-b}, respectively. In addition, we describe compact intertwining relations for multiplication and composition operators in Section \ref{section 5}.

\medskip

Throughout the paper, we write $A \lesssim B$ if there exists an absolute constant $C>0$ such that $A \leq C \cdot B$, and we write $A \approx B$ when $A \lesssim B$ and $B \lesssim A$.

\section{Preliminaries}

\subsection{Compact intertwining relations}

If $X$ and $Y$ are two quasi-Banach spaces, we denote by $\mathscr{B}(X, Y)$ the collection of all bounded linear operators from $X$ to $Y$, and by $\mathcal{K}(X, Y)$ the collection of all compact elements of $\mathscr{B}(X, Y)$, and by $\mathscr{Q}(X, Y)$ the quotient space $\mathscr{B}(X, Y) / \mathcal{K}(X, Y)$.

For $A \in \mathscr{B}(X, X), B \in \mathscr{B}(Y, Y)$ and $T \in \mathscr{B}(X, Y)$,
we say that $T$ \emph{intertwines} $A$ and $B$ in $\mathscr{Q}(X, Y)$ (or $T$ intertwines $A$ and $B$ compactly) if
\begin{align}\label{1}
T A=B T \quad \bmod \mathcal{K}(X, Y) \quad \text { with } \quad T \neq 0.
\end{align}
More intuitively, the compact intertwining relation is explained by the following commutative diagram:
\[\begin{CD} \label{d1}
  X @>{A}>> X\\
   @VV {T} V @VV {T} V\\
    {Y} @>{B}>> {Y}
\end{CD}\quad\quad\mod \mathcal K(X,Y).\]

\medskip

\noindent When $X=Y$ and $A=B$, it is easy to see that the following two assertions are equivalent:
\begin{itemize}
	\item[(i)] $T$ intertwines every $A\in \mathscr B(X)$ compactly.
	\item[(ii)] The inner derivation $D(T):\mathscr B(X)\to\mathscr B(X)$ ranges in the
	ideal of compact operators.
\end{itemize}

From this point of view, we will study the compact intertwining relations for composition operators and Volterra operators between different Hardy spaces, which are then used to obtain our two main results (Theorems \ref{thm-a} and \ref{thm-b}) as direct consequences. In this paper, we also study the compact intertwining relations for composition operators and Volterra operators, and multiplication operators from Hardy spaces to Bergman spaces.

In the series papers \cite{TYZ,TZ1,TZ2}, Yuan, Tong and Zhou firstly investigate the compact intertwining relations on the Bergman spaces, bounded analytic function spaces and Bloch spaces in the unit disk. By continuing this line of work, we characterize the compact intertwining relations for composition operators and Volterra operators between different Hardy spaces. Our main results on the Volterra-type inner derivation on $\mathscr B(\mathcal{H}^p)$ then follow immediately.

\subsection{Background on Volterra and composition operators}

We collect some preliminary lemmas on boundedness and compactness of Volterra  operators and composition operators
in this subsection.
For $0<\beta<\infty$, recall that the weighted Bloch space $\mathcal{B}^{\beta}$ is the space of all $f \in H(\mathbb{D})$ such that
$$
\|f\|_{\mathcal{B}^{\beta}}:=\sup _{z \in \mathbb{D}}\left(1-|z|^{2}\right)^{\beta}\left|f^{\prime}(z)\right|<\infty.
$$
Notice that $\|\cdot\|_{\mathcal{B}^{\beta}}$ is a complete semi-norm on $\mathcal{B}^{\beta}$, which is M\"{o}bius invariant. When equipped with the norm
$$
\|f\|=|f(0)|+\|f\|_{\mathcal{B}^{\beta}},
$$
the weighted Bloch space $\mathcal{B}^{\beta}$ becomes a Banach space. 

Denote by $\mathcal{B}_{0}^{\beta}$ the subspace of $\mathcal{B}^{\beta}$ consisting of those functions $f \in \mathcal{B}^{\beta}$ for which
$$
\lim _{|z| \rightarrow 1}\left(1-|z|^{2}\right)^{\beta}\left|f^{\prime}(z)\right|=0.
$$
For $\delta\geq0$, we define the space $\mathcal{H}^{\infty, \delta}$ of analytic functions by
\begin{align}\label{y1}
\mathcal{H}^{\infty, \delta}=\left\{f \in H(\mathbb{D}):\|f\|_{\infty}=\sup _{z \in \mathbb{D}}\left(1-|z|^{2}\right)^{\delta}|f(z)|<\infty\right\}
\end{align}
and write $\mathcal{H}^{\infty}$ for the space of non-weighted bounded analytic functions $\mathcal{H}^{\infty, 0}$. Further, we let $\mathcal{H}_{0}^{\infty, \delta}$ be the subspace of $\mathcal{H}^{\infty, \delta}$ consisting of $f \in \mathcal{H}^{\infty, \delta}$ with
$$
\lim _{|z| \rightarrow 1}\left(1-|z|^{2}\right)^{\delta}|f(z)|=0.
$$

\begin{remark_v2} It is well-known that, for $\delta>0, \mathcal{H}^{\infty, \delta}=\mathcal{B}^{1+\delta}$ and $\mathcal{H}_{0}^{\infty, \delta}=\mathcal{B}_{0}^{1+\delta}$ (see, for example, Proposition 7 of \cite{Zhu1}).
\end{remark_v2}

Let $p, q$ and $s$ be real numbers such that $0<p<\infty,-2<q<\infty$ and $0<s<\infty$. We say that a function $f\in H(\D)$ belongs to the space $\mathcal{F}(p, q, s)$ if
\begin{equation}\label{40}
\|f\|_{\mathcal{F}(p, q, s)}^{p}:=\sup _{a \in \mathbb{D}} \int_{\mathbb{D}}\left|f^{\prime}(z)\right|^{p}\left(1-|z|^{2}\right)^{q}\left(1-\left|\psi_{a}(z)\right|^{2}\right)^{s} \mathrm{~d} A(z)<\infty.
\end{equation}
The spaces $\mathcal{F}(p, q, s)$ were introduced in \cite{Z1}, and it was shown that many classical function spaces can be identified as $\mathcal{F}(p, q, s)$ with suitable parameters. Further, it was proved in \cite[Theorem 1]{Z2} that, when $-1<\alpha<\infty$, $\mathcal{F}(p, p \alpha-2, s)=\mathcal{B}^{\alpha}$ for every $p>0$ and $s>1$ (see also Theorem $1.3$ of \cite{Z1}). For $s=1$, we define $\mathcal{B M O A}$ type spaces by setting $\mathcal{B M O A}_{p}^{\alpha}=\mathcal{F}(p, p \alpha-2,1)$. It is known that $\mathcal{BMOA}_{2}^{1}=\mathcal{BMOA}$. We recall that the space $\mathcal{V M O A}_{p}^{\alpha}$ consists of those holomorphic functions $f$ in $\D$ with 
\begin{equation}\label{c1}
\lim _{|a|\rightarrow 1} \int_{\mathbb{D}}\left|f^{\prime}(z)\right|^{p}\left(1-|z|^{2}\right)^{p\alpha-2}\left(1-\left|\psi_{a}(z)\right|^{2}\right)=0.
\end{equation}

We now summarize further preliminary results in the following four lemmas. These results are all known or can be obtained with slight modifications of existing results and their proofs. For the first one, see Theorem 5 of \cite{JR}.

\begin{lemma}\label{18}
Let $0<p,q<\infty$, $g \in H(\mathbb{D}),-1<\alpha<\infty,$ and $\gamma=\frac{\alpha+2}{q}-\frac{1}{p}$.
\begin{itemize}
  \item[(i)] If $p<q$ and $\gamma+1 \geq 0$, then $J_{g}: \mathcal{H}^{p} \rightarrow \mathcal{A}_{\alpha}^{q}$ is bounded if and only if $g \in \mathcal{B}^{1+\gamma}$.
  \item[(ii)] If $p=q$, then $J_{g}: \mathcal{H}^{p} \rightarrow \mathcal{A}_{\alpha}^{q}$ is bounded if and only if $$g \in \mathcal{B M O A}_{p}^{1+(\alpha+1) / p}.$$
  \item[(iii)] If $p<q$ and $\gamma+1 \geq 0$, then $J_{g}: \mathcal{H}^{p} \rightarrow \mathcal{A}_{\alpha}^{q}$ is compact if and only if $g \in \mathcal{B}^{1+\gamma}_{0}$.
  \item[(iv)] If $p=q$, then $J_{g}: \mathcal{H}^{p} \rightarrow \mathcal{A}_{\alpha}^{q}$ is compact if and only if $$g \in \mathcal{V M O A}_{p}^{1+(\alpha+1) / p}$$.
\end{itemize}
\end{lemma}

The proof of the following lemma is similar to the proofs of Theorem 5 and Corollary 7 of \cite{JR}.

\begin{lemma}\label{7g}
Let $0<p,q<\infty$, $g \in H(\mathbb{D}),-1<\alpha<\infty,$ and $\gamma=\frac{\alpha+2}{q}-\frac{1}{p}$.
\begin{itemize}
  \item[(i)] If $p\leq q$, then $I_{g}: \mathcal{H}^{p} \rightarrow \mathcal{A}_{\alpha}^{q}$ is bounded if and only if $g \in \mathcal{B}^{1+\gamma}$.
  \item[(ii)] If $p\leq q$, then $I_{g}: \mathcal{H}^{p} \rightarrow \mathcal{A}_{\alpha}^{q}$ is compact if and only if $g \in \mathcal{B}_{0}^{1+\gamma}$.
\end{itemize}
\end{lemma}

The following assertions can be obtained from the main theorems of \cite{PXJ} and \cite{ZRH}.

\begin{lemma}\label{6f}
Let $0<p,q<\infty$, $g \in H(\mathbb{D}),-1<\alpha<\infty,$ and $\gamma=\frac{\alpha+2}{q}-\frac{1}{p}$.
\begin{itemize}
  \item[(i)] If $p<q$ and $\gamma>0$, then $M_{g}: \mathcal{H}^{p} \rightarrow \mathcal{A}_{\alpha}^{q}$ is bounded if and only if $g \in \mathcal{B}^{1+\gamma}$.
  \item[(ii)] If $p<q$ and $\gamma=0$, then $M_{g}: \mathcal{H}^{p} \rightarrow \mathcal{A}_{\alpha}^{q}$ is bounded if and only if $g \in \mathcal{H}^{\infty}$.
  \item[(iii)] If $p=q$, then $M_{g}: \mathcal{H}^{p} \rightarrow \mathcal{A}_{\alpha}^{p}$ is bounded if and only if $g \in \mathcal{B M O A}_{p}^{1+(\alpha+1) / p}$.
  \item[(iv)] If $p<q$ and $\gamma>0$, then $M_{g}: \mathcal{H}^{p} \rightarrow \mathcal{A}_{\alpha}^{q}$ is compact if and only if $g \in \mathcal{B}^{1+\gamma}_{0}$.
  \item[(v)] If $p<q$ and $\gamma=0$, then $M_{g}: \mathcal{H}^{p} \rightarrow \mathcal{A}_{\alpha}^{q}$ is compact if and only if $g\equiv0$.
  \item[(vi)] If $p=q=2$, then $M_{g}: \mathcal{H}^{2} \rightarrow \mathcal{A}_{\alpha}^{2}$ is compact if and only if $g \in \mathcal{V M O A}_{2}^{1+(\alpha+1) / 2}$.
\end{itemize}
\end{lemma}

Finally, the following lemma summarizes the characterizations obtained in \cite{AJ} for boundedness and compactness of the operators $J_{g}$ and $I_{g}$ and the multiplication operator $M_{g}$ acting from $\mathcal{H}^{p}$ to $\mathcal{H}^{q}$.

\begin{lemma}\label{181}
Let $0<p,q<\infty$, $g \in H(\mathbb{D})$. Then
\begin{itemize}
  \item[(i)] If $\frac{q}{q+1}\leq p<q$,  then $J_{g}: \mathcal{H}^{p} \rightarrow \mathcal{H}^{q}$ is bounded if and only if $g \in \mathcal{B}^{1+\frac{1}{q}-\frac{1}{p}}$.
  \item[(ii)] If $p=q$, then $J_{g}: \mathcal{H}^{p} \rightarrow \mathcal{H}^{q}$ is bounded if and only if $g \in \mathcal{B M O A}$.
  \item[(iii)] If $p=q$, then $I_{g}(\mbox{or }M_{g}): \mathcal{H}^{p} \rightarrow \mathcal{H}^{q}$ is bounded if and only if $g \in \mathcal{H}^{\infty}$.
  \item[(iv)] If $\frac{q}{q+1}\leq p<q$, then $J_{g}: \mathcal{H}^{p} \rightarrow \mathcal{H}^{q}$ is compact if and only if $g \in \mathcal{B}^{1+\frac{1}{q}-\frac{1}{p}}_{0}$.
  \item[(v)] If $p=q$, then $J_{g}: \mathcal{H}^{p} \rightarrow \mathcal{H}^{q}$ is compact if and only if $g \in \mathcal{V M O A}$.
  \item[(vi)] If $p=q$, then $I_{g}(\mbox{or }M_{g}): \mathcal{H}^{p} \rightarrow \mathcal{H}^{q}$ is compact if and only if $g \equiv0$.
\end{itemize}
\end{lemma}

\subsection{Carleson measures}
In this subsection, we state the generalized Carleson measure theorem for $\mathcal{H}^{p}$. A classical theorem of Carleson \cite{LC1, LC2} states that the injection map from the Hardy space $\mathcal{H}^{p}$ into the measure space $\mathcal{L}^{p}(d \mu)$ is bounded if and only if the positive measure $\mu$ on $\mathbb{D}$ is a bounded Carleson measure. For $0<s<\infty$, a positive measure $\mu$ on $\mathbb{D}$ is a bounded $s$-Carleson measure if
\begin{align}\label{11}
\|\mu\|_{CM_{s}}:=\sup _{I} \frac{\mu(S(I))}{|I|^{s}}<\infty,
\end{align}
where $|I|$ denotes the arc length of a subarc $I$ of $\mathbb{T}$,
$$
S(I)=\left\{re^{it} \in \mathbb{D}: e^{it} \in I, 1-|I| \leqslant r<1\right\}
$$
is the Carleson box based on $I$, and the supremum is taken over all subarcs $I$ of $\mathbb{T}$ such that $|I| < 1$. We associate to each $a \in \mathbb{D}\backslash\{0\}$ the interval $I_{a}=\left\{\zeta \in \mathbb{T} : |\zeta-\frac{a}{|a|}|\leq \frac{1-|a|}{2}\right\}$, and denote by $S(a)=S(I_a)$. A positive measure $\mu$ on $\mathbb{D}$ is a vanishing $s$-Carleson measure if the limits
\begin{align}\label{C}
\lim_{|I|\rightarrow 0} \frac{\mu(S(I))}{|I|^{s}}=0
\end{align}
hold uniformly for $I \in \mathbb{T}$.
It is well known (see \cite{ASX} and \cite{GJ}) that $\mu$ on $\mathbb{D}$ is an $s$-Carleson measure if and only if
\begin{align}\label{13}
\sup _{a \in \mathbb{D}} \int_{\mathbb{D}}\left(\frac{1-|a|^2}{|1-\overline{a}z|^2}\right)^{s} \mathrm{~d} \mu(z)<\infty.
\end{align}

For $z \in \mathbb{D}$ and $\varphi \in \mathcal{L}^{1}\left(\mathbb{T}\right)$, the Hardy-Littlewood maximal function is defined by 
$$
M(\varphi)(z)=\sup_{I} \frac{1}{|I|} \int_{I}|\varphi(\zeta)| |d \zeta|, \quad z \in \mathbb{D},
$$
where the supremum is taken over all arcs $I\subset \mathbb{T}$ for which $z\in S(I)$.

The following lemma follows from Theorem 2.1 of \cite{LLR1}. We simplify the result in the one-dimensional case as follows.
\begin{lemma}\label{cvb}
	Let $0<p \leq q<\infty$ and $0<\alpha<\infty$ such that p$\alpha>1$. Let $\mu$ be a positive Borel measure on $\mathbb{D}$. 
	Then $\left[M\left((\cdot)^{1 / \alpha}\right)\right]^{\alpha}: \mathcal{L}^{p}\left(\mathbb{T}\right) \rightarrow \mathcal{L}^{q}(\mu)$ is bounded
	 if and only if $\mu$ is a $q / p$-Carleson measure. Moreover, we have
	$$
	\left\|\left[M\left((\cdot)^{1 / \alpha}\right)\right]^{\alpha}\right\|_{\mathcal{L}^{p}\left(\mathbb{T}\right) \rightarrow \mathcal{L}^{q}(\mu)}^{q} \approx \sup _{I} \frac{\mu(S(I))}{|I|^{q / p}}.
	$$
\end{lemma}

The estimates of the next result follow from the results of Luecking---see Theorem 3.1 of \cite{Lue}.

\begin{lemma}\label{12}
Let $\mu$ be a positive Borel measure on $\mathbb{D}$ and k $\in$ $\mathbb{N}$. If either $2 \leq p=q$ or $0<p <q<\infty$, the following conditions are equivalent:
\begin{itemize}
  \item[(i)]
$
  \int_{\mathbb{D}}|f^{(k)}(z)|^{q} \mathrm{~d} \mu(z) \lesssim \|f\|_{\mathcal{H}^{p}}^{q}$
  for all $f \in \mathcal{H}^{p}.$
  \item[(ii)] $\mu(S(I)) \lesssim |I|^{(1+kp)q/p}$ for all $I$.
\end{itemize}
\end{lemma}

The little oh version of the preceding result was obtained in Theorem 1 in \cite{PPP} and can be formulated as follows:

\begin{lemma}\label{2b}
Let $\mu$ be a positive Borel measure on $\mathbb{D}$ and k $\in$ $\mathbb{N}$. If either $2 \leq p=q$ or $0<p <q<\infty$, the following conditions are equivalent:
\begin{itemize}
  \item[(i)] If $\left\{f_{j}\right\}$ is a bounded sequence in $\mathcal{H}^{p}$ and $f_{j}(z) \rightarrow 0$ for every $z \in \mathbb{D}$, then
  \begin{align}
  \lim _{j \rightarrow \infty} \int_{\mathbb{D}}\left|f_{j}^{(k)}(z)\right|^{q} \mathrm{~d} \mu(z)=0 .
  \end{align}
 \item[(ii)] The limits
  \begin{align}
\lim _{|I| \rightarrow 0} \frac{\mu\left(S(I)\right)}{|I|^{(1+kp)q/p}}=0
  \end{align}
hold uniformly for $I \in \mathbb{T}$.
\end{itemize}
\end{lemma}
\begin{remark}\label{Hp}
    For the case $k=0$, Lemma \ref{12} and Lemma \ref{2b} hold true whenever $0<p \leq q<\infty$ (see Theorem 3.4 of \cite{BJ} and Theorem 9.4 of \cite{DP}).
\end{remark}

\subsection{Area operators}
If $\zeta \in \mathbb{T}$ and $\gamma > 2$ are given,  the \emph{Korányi approach region} 
$\Gamma_{\gamma}(\zeta)$ with aperture $\gamma/2$ is defined by
$$
\Gamma(\zeta):= \Gamma_\gamma(\zeta)=\left\{z \in \mathbb{D}:|\zeta-z|<\frac \gamma 2\left(1-|z|\right)\right\}.
$$

For every $z \in \mathbb{D}$, let us denote
$$
I(z)=\{\zeta \in \partial \mathbb{D}: z \in \Gamma(\zeta)\}.
$$
It is clear that $I(z)$ is an open arc on $\partial \mathbb{D}$ with center 
$z /|z|$ whenever $z \neq 0$. Moreover, $|I(z)| \asymp 1-|z|.$

Let $\mu$ be a positive Borel measure on $\mathbb{D}$ and $s>0$. The \emph{area operator} $A_{\mu}^{s}$ acting on $H\left(\mathbb{D}\right)$ is the sublinear operator defined by
$$
A_{\mu}^{s}(f)(\zeta)=\left(\int_{\Gamma(\zeta)}|f(z)|^{s} \frac{\mathrm{d} \mu(z)}{(1-|z|)}\right)^{1 / s}
$$
Area operators are important both in analysis and geometry. They are related to, for example, the nontangential maximal functions,
Littlewood-Paley operators, multipliers, Poisson integrals, and tent spaces. For the study of boundedness and compactness of area operators $A_{\mu}^{s}$ on the Hardy space and weighted Bergman spaces in the unit disk, see \cite{GLW}, \cite{ZW}. The next estimate is the celebrated Calder\'{o}n's area theorem \cite{AC}. The variant we use can be found in \cite[Theorem 3.1]{PJ1} and \cite[Theorem D]{SJA}.

\begin{lemma}\label{sss}
Suppose that $f \in H(\mathbb{D})$ and $0<p<\infty$, then 
$$
\|f\|_{\mathcal{H}^{p}}^{p}	\approx \int_{\mathbb{T}} \left(\int_{\Gamma(\zeta)}|f^{\prime}(z)|^{2} \mathrm{d} A(z)\right)^{p/2} |\mathrm{d} \zeta|.
$$
\end{lemma}

For the proof of the following lemma, see Theorem 3.4 of \cite{GLW}.

\begin{lemma}\label{zzz}
Let $0<p, q, s<\infty$ and $\mu$ be a positive Borel measure on $\mathbb{D}$. If $0<p \leq q<\infty$, then
\begin{itemize}
\item[(i)] $A_{\mu}^{s}: \mathcal{H}^{p} \rightarrow \mathcal{L}^{q}$ is bounded if and only
 if $\mu$ is a $\left(1+\frac{s}{p}-\frac{s}{q}\right)$-Carleson measure;
\item[(ii)] $A_{\mu}^{s}: \mathcal{H}^{p} \rightarrow \mathcal{L}^{q}$ is compact if and only
if $\mu$ is a compact $\left(1+\frac{s}{p}-\frac{s}{q}\right)$-Carleson measure.
\end{itemize}
\end{lemma}

To characterize the compact intertwining relations for the operator $I_g$, we define the following \emph{generalized area operator} which is induced
by a nonnegative measure $\mu$ and $0<s<\infty$:
\[\widetilde{A_{\mu}^{s}}(f)(\zeta)=\left(\int_{\Gamma(\zeta)}|f^{\prime}(z)|^{s} \frac{\mathrm{d} \mu(z)}{(1-|z|)}\right)^{1 / s}.\]

The following lemma is a key tool for the proof of Theorem \ref{thm-b}. 

\begin{lemma}\label{66}
Let $\mu$ be a positive Borel measure on $\mathbb{D}$, finite on compact subsets of $\D$. Suppose that $0<p  \leq q<\infty$ and $0< s<\infty$. We have the following two statements:
\begin{itemize}
\item[(i)] If $\widetilde{A_{\mu}^{s}}: \mathcal{H}^{p} \rightarrow \mathcal{L}^{q}$ is bounded
then $\mu$ is a $\left(1+\frac{s}{p}-\frac{s}{q}+s\right)$-Carleson measure;
\item[(ii)] If $\widetilde{A_{\mu}^{s}}: \mathcal{H}^{p} \rightarrow \mathcal{L}^{q}$ is compact
then $\mu$ is a vanishing $\left(1+\frac{s}{p}-\frac{s}{q}+s\right)$-Carleson measure.
\end{itemize}
Further, if $p, q, s$ satisfy one of the following
conditions:
\begin{itemize}
\item[(a)] $q=s$ and  either $2 \leq p=q <\infty$ or $0<p  <q<\infty$; 
\item[(b)] $q>s>p^2q/(pq+q-p)$.
\end{itemize}
 Then the two necessary conditions above in \emph{(i)} and \emph{(ii)} are also sufficient.
\end{lemma}
\begin{proof}
The proof is similar to the proof of Theorem 4 in \cite{PR}. The difference is that we need to consider the derivatives of the test functions. We provide the details for completeness.

For $a \in \mathbb{D}$, we consider the following test function
\begin{align}\label{tff}
f_{a, p}(z)=\frac{(1-|a|)^{1/p}}{(1-\overline{a}z)^{2/p}}, \quad z \in \overline{\mathbb{D}}. 
\end{align}
By Forelli-Rudin estimates (see \cite{BH}), we see that $f_{a, p} \in \mathcal{H}^{p}$, $\|f_{a, p}\|_{\mathcal{H}^{p}}\approx 1$. In addition, it is easy to see that
$$
|1-\overline{a}z|\approx1-|a|\approx|I_a|, \quad z \in S(a)
$$ and
\begin{align}\label{tf}
	|f'_{a,p}(z)|\approx \frac{1}{(1-|a|)^{1/p+1}}, \quad z \in S(a).
\end{align}
First we prove (i). We start with the case $q=s$. By \eqref{tf} and Fubini's theorem, we get
\begin{align}\label{liu1}
\begin{aligned}
	\frac{\mu(S(a))}{|I_a|^{s/p+s}}&\lesssim \int_{S(a)}|f'_{a,p}(z)|^sd\mu(z)\\
	&\leq \int_{\mathbb{D}}|f'_{a,p}(z)|^sd\mu(z)\\
	&=\int_{\mathbb{T}}\left( \int_{\Gamma(\zeta)}|f_{a, p}^{\prime}(z)|^{s} \frac{d \mu(z)}{\left(1-|z|\right)}\right) |\mathrm{d}\zeta| \\
	&=\left\|\widetilde{A_{\mu}^{s}} f_{a, p}\right\|_{\mathcal{L}^{q}\left(\mathbb{T}\right)}^{q}\lesssim \|f_{a, p}\|_{\mathcal{H}^{p}}^q\lesssim 1,
\end{aligned}
\end{align}
where the first equality comes from the fact that $\int_{\mathbb{T}} \chi_{\Gamma(\zeta)}(z)|\mathrm{d} \zeta| \approx 1-|z|$. Therefore, $\mu$ is a $\left(\frac{s}{p}+s\right)$-Carleson measure.\\
Next, we consider the case $q>s$.  We can see that
\begin{align}\label{liu11}
\left|f_{a,(q / s)^{\prime}}(\zeta)\right| \approx \frac{1}{(1-|a|^2)^{(1-s / q)}}
\end{align} as $\zeta \in I(z)$, and $z \in S(a)$.
Note that $\widetilde{A_{\mu}^{s}}: \mathcal{H}^{p} \rightarrow \mathcal{L}^{q}$ is bounded. By \eqref{tf}, \eqref{liu11}, Fubini's theorem and H\"{o}lder inequality, we obtain that
\begin{align}\label{liu12}
\begin{aligned}
	\frac{\mu(S(a))}{|I_a|^{1+s / p-s / q+s}} & \lesssim  \int_{S(a)}  |f'_{a, p}(z)|^s\left(\frac{1}{(1-|z|)} \int_{I(z)}\left|f_{a,(q / s)^{\prime}}(\zeta)\right| |d\zeta|\right) d \mu(z) \\
	& \leq \int_{\mathbb{T}}\left|f_{a,(q / s)^{\prime}}(\zeta)\right|\left(\int_{\Gamma(\zeta)}\left|f'_{a, p}(z)\right|^s \frac{d \mu(z)}{(1-|z|)}\right) |d\zeta| \\
	& \leq\left\|f_{a,(q / s)^{\prime}}\right\|_{\mathcal{H}^{(q / s)^{\prime}}}\left\|\widetilde{A_{\mu}^{s}} f_{a, p}\right\|_{\mathcal{L}^{q}\left(\mathbb{T}\right)}^s \lesssim \|f_{a, p}\|_{\mathcal{H}^{p}}^s\lesssim 1.
\end{aligned}
\end{align}
Thus $\mu$ is a $(1+s / p-s / q+s)$-Carleson measure.\\
Finally,  we consider the case $q<s$. Let $1<\beta<\alpha$ satisfing $\frac{\beta}{\alpha}=\frac{q}{s}$. By Fubini's theorem and H\"{o}lder's inequality, we get
\begin{align}\label{LLR1}
\begin{aligned}
 \mu(S(a))&=\int_{\mathbb{D}} \I_{S(a)}(z) d \mu(z) \\
& \approx(1-|a|)^{ s / p+s} \int_{\mathbb{D}} \I_{S(a)}(z)\left|f'_{a, p}(z)\right|^s \frac{1}{(1-|z|)} \int_{I(z)} |d\zeta| d \mu(z) \\
& =(1-|a|)^{ s / p+s} \int_{\mathbb{T}}\left(\int_{\Gamma(\zeta)} \I_{S(a)}(z)\left|f'_{a, p}(z)\right|^s \frac{d \mu(z)}{(1-|z|)}\right)^{1 / \alpha+1 / \alpha^{\prime}} |d\zeta| \\
& \leq(1-|a|)^{ \frac{s / p+s}{\alpha}}\left(\int_{\mathbb{T}}\left(\int_{\Gamma(\zeta)}\left|f'_{a, p}(z)\right|^s \frac{d \mu(z)}{(1-|z|)}\right)^{\beta / \alpha} |d\zeta|\right)^{1 / \beta}  \\
& \qquad \times\left(\int_{\mathbb{T}}\left(\int_{\Gamma(\zeta)} \I_{S(a)}(z) \frac{d \mu(z)}{(1-|z|)}\right)^{\beta^{\prime} / \alpha^{\prime}} |d \zeta|\right)^{1 / \beta^{\prime}}.
\end{aligned}
\end{align}
According to the estimate in (3.6) of \cite{LLR1}, we can estimate the second factor on the right side of the  above inequality as follows:
\begin{align}
\begin{aligned}
&\left(\int_{\mathbb{T}}\left(\int_{\Gamma(\zeta)}\, \I_{S(a)}(z) \frac{d \mu(z)}{(1-|z|)}\right)^{\beta^{\prime} / \alpha^{\prime}} |d \zeta|\right)^{\beta / \beta^{\prime}}\\
&\leq\mu(S(a))^{\beta / \beta^{\prime}} \left(\sup _{z \in \mathbb{D}} \frac{\mu(S(a) \cap S(z))}{(1-|z|)}\right)^{\beta / \alpha^{\prime}-\beta / \beta^{\prime}}\\
&=\mu(S(a))^{\beta-1}\left(\sup _{z \in \mathbb{D}} \frac{\mu(S(a) \cap S(z))}{(1-|z|)}\right)^{\beta / \alpha^{\prime}-\beta / \beta^{\prime}}.
\end{aligned}
\end{align}
Inserting this into \eqref{LLR1}, we have 
\begin{align}\label{A}
\mu(S(a))\lesssim (1-|a|)^{ (s / p+s)\frac{\beta}{\alpha}}\left\|\widetilde{A_{\mu}^{s}} f_{a, p}\right\|_{\mathcal{L}^{q}}^q \cdot \left(\sup _{z \in \mathbb{D}} \frac{\mu(S(a) \cap S(z))}{(1-|z|)}\right)^{\beta / \alpha^{\prime}-\beta / \beta^{\prime}}.
\end{align}
We define $d\mu_r(z)=\I_{D(0, r)}d\mu(z)$ for $0<r<1$. It is easy to see that  $$\left\|\widetilde{A_{\mu_r}^{s}}\right\|_{\mathcal{H}^p \rightarrow \mathcal{L}^{q}} \leq\left\|\widetilde{A_{\mu}^{s}}\right\|_{\mathcal{H}^p \rightarrow \mathcal{L}^{q}}.$$
Putting these estimates together, we have
$$
\begin{aligned}
&\frac{\mu_r(S(a))}{|I_a|^{1+s / p-s / q+s}} \\
& \quad \lesssim\left\|\widetilde{A_{\mu_r}^{s}} f_{a, p}\right\|_{\mathcal{L}^{q}}^q(1-|a|)^{\frac{q}{p}+q-1-s / p+s / q-s}\left(\sup _{z \in \mathbb{D}} \frac{\mu_r(S(a) \cap S(z))}{(1-|z|)}\right)^{\beta / \alpha^{\prime}-\beta / \beta^{\prime}} \\
& \quad =\left\|\widetilde{A_{\mu_r}^{s}} f_{a, p}\right\|_{\mathcal{L}^{q}}^q\left(\sup _{z: S(z) \subset S(a)} \frac{\mu_r(S(a) \cap S(z))}{(1-|a|^2)^{s(q-p+pq)  /(p q)}(1-|z|)}\right)^{1-q / s} \\
& \quad \lesssim\left\|\widetilde{A_{\mu_r}^{s}} f_{a, p}\right\|_{\mathcal{L}^{q}}^q\left(\sup _{z: S(z) \subset S(a)} \frac{\mu_r(S(z))}{(1-|z|)^{(1+s / p-s / q+s) }}\right)^{1-q / s}, \quad a \in \mathbb{D}.
\end{aligned}
$$
Consequently,
$$
\begin{aligned}
&\sup_{a \in \mathbb{D}}\frac{\mu_r(S(a))}{|I_a|^{1+s / p-s / q+s}}\\
& \quad\lesssim \left\|\widetilde{A_{\mu}^{s}}\right\|^q_{\mathcal{H}^p \rightarrow \mathcal{L}^{q}}\left(\sup_{a \in \mathbb{D}}\sup _{z: S(z) \subset S(a)} \frac{\mu_r(S(z))}{(1-|z|)^{(1+s / p-s / q+s) }}\right)^{1-q / s}\\
& \quad =\left\|\widetilde{A_{\mu}^{s}}\right\|^q_{\mathcal{H}^p \rightarrow \mathcal{L}^{q}}\left(\sup_{a \in \mathbb{D}}\frac{\mu_r(S(a))}{(1-|a|)^{(1+s / p-s / q+s) }}\right)^{1-q / s}.
\end{aligned}
$$
Therefore,
$$
\sup _{a \in \mathbb{D}, r\in(0,1)} \frac{\mu_r(S(a))}{|I_a|^{1+s / p-s / q+s}} \lesssim\left\|\widetilde{A_{\mu}^{s}}\right\|_{\mathcal{H}^p \rightarrow \mathcal{L}^{q}} ^s .
$$
It then follows from Fatou's lemma by letting $r\to 1^-$ that
\begin{align}\label{B}
\sup _{a \in \mathbb{D}} \frac{\mu(S(a))}{|I_a|^{1+s / p-s / q+s}} \lesssim\left\|\widetilde{A_{\mu}^{s}}\right\|_{\mathcal{H}^p \rightarrow \mathcal{L}^{q}}^s .
\end{align}
Thus $\mu$ is a $(1+s / p-s / q+s)$-Carleson measure.

Now we prove (ii). Suppose $\widetilde{\mathcal{A}^{s}_{\mu}}: \mathcal{H}^{p} \rightarrow \mathcal{L}^{q}$ is compact. For each $a \in \mathbb{D}$, let $f_{a, p}$ be given by \eqref{tff}. It is noted that $\left\|f_{a, p}\right\|_{\mathcal{H}^{p}} \lesssim 1$ uniformly in $a \in \mathbb{D}$  and $\left\{f_{a, p}\right\}$ converges uniformly to zero on compact subsets of $\mathbb{D}$, as $|a| \rightarrow 1$. Hence
$$
\lim _{|a| \rightarrow 1}\left\|\widetilde{\mathcal{A}^{s}_{\mu}}\left(f_{a, p}\right)\right\|_{\mathcal{L}^{q}}=0 .
$$
Inequalities \eqref{liu1}, \eqref{liu12} and \eqref{B} imply that
$$
\lim _{|a| \rightarrow 1} \frac{\mu(S(a))}{|I_a|^{1+s / p-s / q+s}} =0 .
$$
Thus $\mu$ is a vanishing $(1+s/p-s/q+s)$-Carleson measure.

To prove sufficiency in our last assertion, suppose first that condition $(a)$ holds in our assumption.
By Fubini's theorem and Lemma \ref{12}, we get
\begin{align}\label{V11}
\begin{aligned}
\left\|\widetilde{A_{\mu}^{s}} f\right\|_{\mathcal{L}^{q}\left(\mathbb{T}\right)}^{q}
&= \int_{\mathbb{T}}\left( \int_{\Gamma(\zeta)}|f^{\prime}(z)|^{s} \frac{d \mu(z)}{\left(1-|z|\right)}\right) |\mathrm{d}\zeta| \\
& =\int_{\mathbb{D}}|f^{\prime}(z)|^{s} \mathrm{d} \mu(z) \\
& \leq\|\mu\|_{CM_{(1+p)\frac{s}{p}}}\|f\|_{\mathcal{H}^{p}}^{q}.
\end{aligned}
\end{align}
Hence $\widetilde{A_{\mu}^{s}}:\mathcal H^p\to\mathcal L^q$ is bounded.

Suppose next that $q>s> \frac{p^2q}{pq+q-p}$. Let $t=s+(1-s/q)/(1+1/p)>s$.
Then $p<t$ and $(t/s)'/(q/s)'=s+1+s/p-s/q$. Let $M$ be the Hardy-Littlewood maximal function. Then by Lemma \ref{cvb} we have
\begin{align}\label{mqts}
\|M\|_{\mathcal L^{(q/s)'}(\mathbb{T})\rightarrow \mathcal L^{(t/s)'}(\mu)}^{(t/s)'}\approx \sup_{a\in\D}\frac{\mu(S(a))}{(1-|a|)^{s+1+s/p-s/q}}.
\end{align}
Thus, by duality's theorem,  Fubini's Theorem, H\"{o}lder's inequality and  \eqref{mqts}, we have that
\begin{align}\label{CM11}
\begin{aligned}
	\left\|\widetilde{ A^{s}_{\mu}}f\right\|_{\mathcal{L}^{q}(\mathbb{T})}^{s}=&\left\|(\widetilde{ A^{s}_{\mu}}f)^s\right\|_{\mathcal{L}^{q/s}(\mathbb{T})}\\
 \leq&\sup _{\|h\|_{\mathcal{L}^{\left(q / s\right)^{\prime}}} \leq 1} \int_{\mathbb{T}}|h(\zeta)|\left(\int_{\Gamma(\zeta)}\left|f^{\prime}(z)\right|^{s} \frac{\mathrm{d} \mu(z)}{1-|z|}\right)  |\mathrm{d}\zeta|\\
	=& \sup _{\|h\|_{\mathcal{L}^{\left(q / s\right)^{\prime}}} \leq 1}\int_{\mathbb{D} }\left|f^{\prime}(z)\right|^{s} \left(\frac{1}{|I|} \int_{I(z)}|h(\zeta)| |\mathrm{d} \zeta|\right) \mathrm{d} \mu(z)\\
	\leq& \sup _{\|h\|_{\mathcal{L}^{\left(q / s\right)^{\prime}}} \leq 1}\|f^{\prime}\|^{s}_{\mathcal{L}^{t}(\mu)}\|M(h)\|_{\mathcal{L}^{(t / s)^{\prime}}(\mu)}\\
	\leq &\sup _{\|h\|_{\mathcal{L}^{\left(q / s\right)^{\prime}}} \leq 1}\|D\|^s_{\mathcal H^p \rightarrow \mathcal{L}^t(\mu)}\|f\|^s_{\mathcal H^p} \|M\|_{\mathcal{L}^{(q / s)^{\prime}}\rightarrow \mathcal{L}^{(t / s)^{\prime}}(\mu)}\|h\|_{\mathcal{L}^{(q / s)^{\prime}}}\\
	 \lesssim&\left(\sup_{a\in\D}\frac{\mu(S(a))}{(1-|a|)^{s+1+s/p-s/q}}\right)^{s/t+1/(t/s)'}\|f\|^s_{\mathcal H^p}.
\end{aligned}
\end{align}
where $D$ denotes the differentiation operator. Hence $\widetilde{A_{\mu}^{s}}:\mathcal H^p\to\mathcal L^q$ is bounded.

Next we consider compactness. Suppose that $q>s> \frac{p^2q}{pq+q-p}$ and let $\mu$ be a vanishing $(s+1+s/p-s/q)$-Carleson measure. Then $\mu$ is a $(s+1+s/p-s/q)$-Carleson measure, and so a finite measure in $\D$. Let $t=s+(1-s/q)/(1+1/p)>s$, then $(1+p)t/p=s+1+s/p-s/q$. By Lemma \ref{2b}, $D: \mathcal H^p \rightarrow \mathcal L^t(\mu)$ is compact. Let $\left\{f_{n}\right\}$ be a bounded sequence in $\mathcal{H}^{p}$. Then we can choose a subsequence $\left\{f_{n_{k}}\right\}$  that uniformly on compact subsets to some $f \in \mathcal{H}^{p}$ and a subsequence $\left\{f'_{n_{k}}\right\}$ that converges in $\mathcal{L}^{t}(\mu)$. Write $g_{n_{k}}=f_{n_{k}}-f$. By \eqref{CM11}, we have
$$
\left\|\widetilde{ A^{s}_{\mu}}(g_{n_{k}})\right\|_{\mathcal{L}^{q}}^{s}\lesssim\|g_{n_{k}}^{\prime}\|^{s}_{\mathcal{L}^{t}(\mu)}.
$$
Then 
\begin{align}\label{C11}
\lim _{k \rightarrow \infty}\left\|\widetilde{ A^{s}_{\mu}}\left(g_{n_{k}}\right)\right\|_{\mathcal{L}^{q}}^{s}=0 .
\end{align}
Two applications of Minkowski's inequality gives
$$
\left|||\widetilde{A^{s}_{\mu}}\left(f_{n_{k}}\right)||_{\mathcal{L}^{q}}-||\widetilde{A^{s}_{\mu}}(f)||_{\mathcal{L}^{q}}\right| \leq\left\|\widetilde{A^{s}_{\mu}}\left(g_{n_{k}}\right)\right\|_{\mathcal{L}^{q}} \rightarrow 0, \quad k \rightarrow \infty .
$$
Together with \eqref{C11}, we get
$$
\lim _{k \rightarrow \infty}\left\|\widetilde{ A^{s}_{\mu}}\left(f_{n_{k}}\right)\right\|_{\mathcal{L}^{q}}=\left\|\widetilde{ A^{s}_{\mu}}\left(f\right)\right\|_{\mathcal{L}^{q}}.
$$
Moreover, since $\left\{f_{n_{k}}\right\}$ converges uniformly on compact subsets of $\mathbb{D}$ to $f$, by Theorem $5.3$ in \cite{ES}, then $\varphi_{k}(\zeta)=$ $\left(\int_{\Gamma(\zeta)}\left|f^{\prime}_{n_{k}}(z)\right|^{s} \frac{\mathrm{d} \mu{(z))}}{1-|z|}\right)^{1 / s}$ converges to 
$$\varphi(\zeta)=\left(\int_{\Gamma(\zeta)}|f^{\prime}(z)|^{s} \frac{\mathrm{d} \mu(z)}{1-|z|}\right)^{1 / s}$$ for each $\zeta \in \mathbb{T}$. Therefore by Lemma $1$ in \cite{DP} yields
$$
\lim _{k \rightarrow \infty}\left\|\widetilde{A^{s}_{\mu}}\left(f_{n_{k}}\right)-\widetilde{A^{s}_{\mu}}(f)\right\|_{\mathcal{L}^{q}}=\lim _{k \rightarrow \infty}\left\|\varphi_{k}-\varphi\right\|_{\mathcal{L}^{q}}=0,
$$
and thus $\widetilde{A^{s}_{\mu}}: \mathcal{H}^{p} \rightarrow \mathcal{L}^{q}$ is compact.\\
Suppose that condition $(a)$ holds, by \eqref{V11}, we also have that
\begin{align}\label{}
\lim _{k \rightarrow \infty}\left\|\widetilde{ A^{s}_{\mu}}\left(g_{n_{k}}\right)\right\|_{\mathcal{L}^{q}}=0,
\end{align}
by arguing as in the previous case we see that $\widetilde{A^{s}_{\mu}}: \mathcal{H}^{p} \rightarrow \mathcal{L}^{q}$ is compact.

\end{proof}

\section{Proof of Theorem \ref{thm-a}}\label{proof-of-thm-a}

In this section, we first consider the compact intertwining relation for $J_{g}$ and $C_{\varphi}$ from $\mathcal{H}^{p}$  to $\mathcal{H}^{q}$. The proof of Theorem~\ref{thm-a} then
follows immediately as a corollary. We also consider the compact intertwining relation for $J_{g}$ and $C_{\varphi}$ from $\mathcal{H}^{p}$  to $\mathcal{A}_{\alpha}^{q}$ at the end of
this section.

To prove Theorem \ref{thm-a}, for $\varphi\in S(\D)$ and $g,h\in H(\D)$, we consider the following operator
$$
\begin{aligned}
T_{\varphi, g, h}f(z)
&=\int_{0}^{\varphi(z)} f(w) g^{\prime}(w) \mathrm{d} w-\int_{0}^{z} f(\varphi(w)) h^{\prime}(w) \mathrm{d} w
\end{aligned}
$$
for $f\in \mathcal{H}^p$ and $z\in\D$.

To characterize the properties of $T_{\varphi,g, h}$, we define another integral operator as follows:
$$
\mathrm{I}_{\varphi}^{p, q}(u)(a)=\int_{\mathbb{D}}\left(\frac{1-|a|^{2}}{|1-\bar{a} \varphi(z)|^{2}}\right)^{1+\frac{2}{p}-\frac{2}{q}}|u(z)|^{2} (1-|\varphi(z)|^{2}) \mathrm{~d} A(z) .
$$

\begin{proposition}\label{yyy}
Let $0<p\leq q<\infty$. Assume that $\varphi \in$ $Aut(\mathbb{D})$,  and $g,h\in H(\mathbb D)$.
\begin{itemize}
\item[(i)] $T_{\varphi, g, h}$ is a bounded operator from $\mathcal{H}^{p}$ to $\mathcal{H}^{q}$ if and only if
\begin{align}
\sup _{a \in \mathbb{D}} \mathrm{I}_{\varphi}^{p, q}\left((g \circ \varphi-h)^{\prime}\right)(a)<\infty .
\end{align}
\item[(ii)]$T_{\varphi, g, h}$ is a compact operator from $\mathcal{H}^{p}$ to $\mathcal{H}^{q}$ if and only if $T_{\varphi, g, h}$ is bounded and
\begin{align}
\lim _{|a| \rightarrow 1} \mathrm{I}_{\varphi}^{p, q}\left((g \circ \varphi-h)^{\prime}\right)(a)=0 .
\end{align}
\end{itemize}
\end{proposition}

\begin{proof}
First, by Lemma \ref{sss}, since $(T_{\varphi, g, h} f)' (z)= (g\circ \varphi-h)'(z) f(\varphi(z))$,
\begin{align}\label{gamma}
\begin{aligned}
\|T_{\varphi, g, h} f\|_{\mathcal{H}^{q}}^{q}
 \approx \int_{\mathbb{T}} \left(\int_{\Gamma(\zeta)}|(g \circ \varphi-h)^{\prime}(z) C_{\varphi}(f)(z)|^{2} \mathrm{d} A(z)\right)^{q/2} |\mathrm{d} \zeta|.
\end{aligned}
\end{align}
Since $\varphi \in$ $Aut(\mathbb{D})$, there is a point $b \in \D$ such that
$$
\varphi(z)=e^{i\theta}\psi_b(z)=e^{i\theta}\frac{b-z}{1-\overline{b}z}, \quad  \theta \in [0, 2\pi).
$$
For $z \in \Gamma(\zeta)$ with $\zeta_1=\psi_b(\zeta)$, we have
\begin{align}\label{G}
\begin{aligned}
    |\zeta_1-\psi_b(z)|
    &=|\psi_b(\zeta)-\psi_b(z)|\\
    &=\left|\frac{1-|b|^2}{(1-\overline{b}\zeta)(1-\overline{b}z)}\right||\zeta-z|\\
    &<\frac{\gamma}{2}\left|\frac{(1-|b|^2)(1-|z|)}{(1-\overline{b}\zeta)(1-\overline{b}z)}\right|\\
    &< \frac{2\gamma}{1-|b|}(1-|\psi_b(z)|).
\end{aligned}
\end{align}
 Set $\eta=e^{i\theta}\zeta_1.$ By \eqref{G}, we have
$$
 |\eta-\varphi(z)|<\frac{2\gamma}{1-|b|}(1-|\varphi(z)|).
$$
Thus, $\varphi(z) \in \Gamma_{\gamma'}(\eta)$ with $\gamma'=\frac{4\gamma}{1-|b|}>1$. Let $w=\varphi(z)$, $\mu_{\varphi}=\nu_{\varphi} \circ \varphi^{-1}$, and $\mathrm{d} \nu_{\varphi}(z)
=\left|(g \circ \varphi-h)^{\prime}(z)\right|^{2}d A(z)$ in \eqref{gamma}. Then
$$
\begin{aligned}
\|T_{\varphi, g, h} f\|_{\mathcal{H}^{q}}^{q}
& \approx \int_{\mathbb{T}} \left(\int_{\Gamma_{\gamma'}(\eta)}|f(w)|^{2} \frac{(1-|w|)\mathrm{d}\mu_{\varphi}(w)}{1-|w|^2}\right)^{q/2} |\mathrm{d} \eta| \\
& =\left\| \left(\int_{\Gamma_{\gamma'}(\eta)}|f(w)|^{2} \frac{\mathrm{d}\mu(w)}{1-|w|^2}\right)^{1/2}\right\|_{\mathcal{L}^{q}(\mathbb{T})}^{q}\\
&=\left\|A_{\mu}^{2}(f)\right\|_{\mathcal{L}^{q}(\mathbb{T})}^{q},
\end{aligned}
$$
where $\mathrm{d}\mu(w)=(1-|w|^2)\mathrm{d}\mu_{\varphi}(w)$ and the 
last identity follows from the observation on page 3 of \cite{PL}.

Hence $T_{\varphi, g, h}: \mathcal{H}^{p} \rightarrow \mathcal{H}^{q}$ is bounded if and only if $A_{\mu}^{2}: \mathcal{H}^{p} \rightarrow \mathcal{L}^{q}\left(\mathbb{T}\right)$ is bounded. Thus, Lemma \ref{zzz} (i) implies that $\mathrm{d} \mu$ is a $(1+\frac{2}{p}-\frac{2}{q})$-Carleson measure. By \eqref{13}, this is equivalent to
$$
\sup _{a \in \mathbb{D}} \int_{\mathbb{D}}\left(\frac{1-|a|^{2}}
{|1-\bar{a} w|^{2}}\right)^{1+\frac{2}{p}-\frac{2}{q}}(1-|w|^2) \mathrm{d} \mu_{\varphi}(w)<\infty.
$$
Changing the variable back to $z$ proves (i).

Similarly we can prove (ii) using Lemma \ref{zzz} (ii). We omit the details.
\end{proof}

Now we are ready to characterize the compact intertwining relation for $J_{g}$ and $C_{\varphi}$ from $\mathcal{H}^{p}$  to $\mathcal{H}^{q}$.
\begin{theorem}\label{567}
Let $0<p\leq q<\infty.$
\begin{itemize}
\item[(i)] If $\frac q{q+1}\leq p<q$, then $J_{g}: \mathcal{H}^{p} \rightarrow \mathcal{H}^{q}$
 compactly intertwines all composition operators $C_{\varphi}$ which are bounded both on $\mathcal{H}^{p}$
 and $\mathcal{H}^{q}$ if and only if $g$ $\in$ $\mathcal{B}^{1+\frac{1}{q}-\frac{1}{p}}_{0}$;
\item[(ii)] If $p=q$, then $J_{g}: \mathcal{H}^{p} \rightarrow \mathcal{H}^{p}$ compactly
intertwines all composition operators $C_{\varphi}$ which are bounded on $\mathcal{H}^{p}$
if and only if $g$ $\in$ $\mathcal{VMOA}$.
\end{itemize}
\end{theorem}

\begin{proof}
First, we consider (i). If $g$ $\in$ $\mathcal{B}^{1+\frac{1}{q}-\frac{1}{p}}_{0}$, the operator
 $J_{g}$ is compact from $\mathcal{H}^{p}$  to $\mathcal{H}^{q}$ by Lemma \ref{181}. Hence
$$
\left.\left.C_{\varphi}\right|_{\mathcal{H}^{q}} J_{g}\right|_{\mathcal{H}^{p} \rightarrow
	\mathcal{H}^{q}}-\left.\left.J_{g}\right|_{\mathcal{H}^{p} \rightarrow \mathcal{H}^{q}} C_{\varphi}\right|_{\mathcal{H}^{p}}
$$
is compact for every $C_{\varphi}$ bounded both on $\mathcal{H}^{p}$ and $\mathcal{H}^{q}$.

For the necessary part of (i), by (i) of Lemma \ref{181}, $J_{g}$ is bounded from $\mathcal{H}^{p}$  to $\mathcal{H}^{q}$ if and only if $g$ $\in$ $\mathcal{B}^{1+\frac{1}{q}-\frac{1}{p}}$. Putting $\varphi(z)=e^{\mathbf{i} \theta} z$ for $\theta \in$ $[0,2 \pi]$, by Proposition \ref{yyy} (ii),  we have
\begin{align}\label{19}
\begin{aligned}
0=&\lim _{|a| \rightarrow 1} \int_{\mathbb{D}}\left(\frac{1-|a|^{2}}{\left|1-\bar{a} e^{\mathbf{i} \theta} z\right|^{2}}\right)^{1+\frac{2}{p}-\frac{2}{q}}\left|e^{\mathbf{i} \theta} g^{\prime}\left(e^{\mathbf{i} \theta} z\right)-g^{\prime}(z)\right|^{2}(1-|z|^{2}) \mathrm{~d} A(z) \\
=&\lim _{|a| \rightarrow 1}\int_{\mathbb{D}}\left(1-|\psi_{a}(z)|^{2}\right)^{1+\frac{2}{p}-\frac{2}{q}}\left|e^{\mathbf{i} \theta} g^{\prime}\left(e^{\mathbf{i} \theta} z\right)-g^{\prime}(z)\right|^{2}(1-|z|^{2})^{\frac{2}{q}-\frac{2}{p}} \mathrm{~d} A(z) \\
\approx & \lim _{|a| \rightarrow 1}(1-|a|^{2})^{1+\frac 1q-\frac 1p}\left|e^{\mathbf{i} \theta} g^{\prime}\left(e^{\mathbf{i} \theta} a\right)-g^{\prime}(a)\right| ,
\end{aligned}
\end{align}
where the last line follows from \eqref{40}.

We estimate the upper bound of last formula in \eqref{19} as follows
$$
\begin{aligned}
&(1-|z|^{2})^{1+\frac 1q-\frac 1p}\left|e^{\mathbf i \theta} g^{\prime}\left(e^{\mathbf{i} \theta} z\right)-g^{\prime}(z)\right| \\
 \leq &(1-|e^{\mathbf{i} \theta }z|^{2})^{1+\frac 1q-\frac 1p}\left|g^{\prime}\left(e^{\mathbf{i} \theta} z\right)\right|+(1-|z|^{2})^{1+\frac 1q-\frac 1p}\left|g^{\prime}(z)\right|  \\
\leq& 2\|g\|_{\mathcal{B}^{1+\frac 1q-\frac 1p}}<\infty,
\end{aligned}
$$
so the upper bound for the estimate in \eqref{19} is independent of $\theta$. 

We write $g(z)=\sum_{n=0}^{\infty} a_{n} z^{n}$ and integrate the right-hand side of \eqref{19} with respect to $\theta$ from 0 to $2 \pi$ as follows
\begin{align}\label{31}
\begin{aligned}
0 &=\int_{0}^{2 \pi} \lim _{|z| \rightarrow 1}(1-|z|^{2})^{1+\frac 1q-\frac 1p}\left|e^{\mathbf{i} \theta} g^{\prime}\left(e^{\mathbf{i} \theta} z\right)-g^{\prime}(z)\right| \mathrm{d} \theta \\
&=\lim _{|z| \rightarrow 1} \int_{0}^{2 \pi}(1-|z|^{2})^{1+\frac 1q-\frac 1p}\left|e^{\mathbf{i} \theta} g^{\prime}\left(e^{\mathbf{i} \theta} z\right)-g^{\prime}(z)\right| \mathrm{d} \theta \\
&=\lim _{|z| \rightarrow 1} (1-|z|^{2})^{1+\frac 1q-\frac 1p}\int_{0}^{2 \pi}\left|\sum_{n=1}^{\infty} n a_{n} z^{n-1}\left(e^{\mathrm{i} n \theta}-1\right)\right| \mathrm{d} \theta \\
& \geq \lim _{|z| \rightarrow 1} (1-|z|^{2})^{1+\frac 1q-\frac 1p}\left|\sum_{n=1}^{\infty} n a_{n} z^{n-1} \int_{0}^{2 \pi}\left(e^{\mathrm{i} n \theta}-1\right) \mathrm{d} \theta\right|\\
&=2 \pi  \lim _{|z| \rightarrow 1}(1-|z|^{2})^{1+\frac 1q-\frac 1p}\left|g^{\prime}(z)\right|,
\end{aligned}
\end{align}
where the dominated convergence theorem is applied to the second line. Thus $g \in$ $\mathcal{B}^{1+\frac 1q-\frac 1p}_{0}$.

Next, we consider (ii). If $g$ $\in$ $\mathcal{VMOA}$, we can see $J_{g}$ is compact on $\mathcal{H}^{p}$ by Lemma \ref{181}. Hence
$$
\left.\left.C_{\varphi}\right|_{\mathcal{H}^{p}} J_{g}\right|_{\mathcal{H}^{p} \rightarrow \mathcal{H}^{p}}-\left.\left.J_{g}\right|_{\mathcal{H}^{p} \rightarrow \mathcal{H}^{p}} C_{\varphi}\right|_{\mathcal{H}^{p}}
$$
is compact for every $C_{\varphi}$ bounded on $\mathcal{H}^{p}$.

For the necessary part of (ii), by (ii) of Lemma \ref{181}, $J_{g}$ is bounded on $\mathcal{H}^{p}$
if and only if $g$ $\in$ $\mathcal{BMOA}$. Putting $\varphi(z)=e^{\mathbf{i} \theta} z$ for $ \theta \in$ $[0,2 \pi]$, by Proposition \ref{yyy} (ii), we have
\begin{align}\label{30}
\begin{aligned}
0=&\lim _{|a| \rightarrow 1} \int_{\mathbb{D}}\frac{1-|a|^{2}}{\left|1-\bar{a} e^{\mathbf{i} \theta} z\right|^{2}}\left|e^{\mathbf{i} \theta} g^{\prime}\left(e^{\mathbf{i} \theta} z\right)-g^{\prime}(z)\right|^{2}(1-|z|^{2}) \mathrm{~d} A(z) \\
=&\lim _{|a| \rightarrow 1}\int_{\mathbb{D}}(1-|\psi_{a}(z)|^{2})\left|e^{\mathbf{i} \theta} g^{\prime}\left(e^{\mathbf{i} \theta} z\right)-g^{\prime}(z)\right|^{2}  \mathrm{~d} A(z).
\end{aligned}
\end{align}
Similarly to (i), we obtain an upper bound for \eqref{30} as follows
$$
\begin{aligned}
&\int_{\mathbb{D}}(1-|\psi_{a}(z)|^{2})\left|e^{\mathbf{i} \theta} g^{\prime}\left(e^{\mathbf{i} \theta} z\right)-g^{\prime}(z)\right|^{2} \mathrm{~d} A(z) \\
 \leq &\int_{\mathbb{D}}\left(1-|\psi_{a}(e^{\mathbf{i} \theta }z)|^{2}\right)\left| g^{\prime}\left(e^{\mathbf{i} \theta} z\right)\right|^{2} \mathrm{~d} A(z) \\ &+\int_{\mathbb{D}}\left(1-|\psi_{a}(z)|^{2}\right)\left|g^{\prime}(z)\right|^{2} \mathrm{~d} A(z)  \\
\leq& 2\|g\|_{\mathcal{B M O A}}<\infty,
\end{aligned}
$$
where the last line follows from \eqref{40}. Hence, \eqref{30} has an upper bound independent of $\theta$.

Again, write $g(z)=\sum_{n=0}^{\infty} a_{n} z^{n}$, and integrate the right-hand side of \eqref{30} with respect to $\theta$ from 0 to $2 \pi$ as follows
$$
\begin{aligned}
0 &=\int_{0}^{2 \pi} \lim _{|a| \rightarrow 1}\int_{\mathbb{D}}(1-|\psi_{a}(z)|^{2})\left|e^{\mathbf{i} \theta} g^{\prime}\left(e^{\mathbf{i} \theta} z\right)-g^{\prime}(z)\right|^{2} \mathrm{~d} A(z) \mathrm{d} \theta \\
&=\lim _{|a| \rightarrow 1} \int_{0}^{2 \pi}\int_{\mathbb{D}}(1-|\psi_{a}(z)|^{2})\left|e^{\mathbf{i} \theta} g^{\prime}\left(e^{\mathbf{i} \theta} z\right)-g^{\prime}(z)\right|^{2} \mathrm{~d} A(z) \mathrm{d} \theta \\
&=\lim _{|a| \rightarrow 1} \int_{\mathbb{D}}(1-|\psi_{a}(z)|^{2})\int_{0}^{2 \pi}\left|\sum_{n=1}^{\infty} n a_{n} z^{n-1}\left(e^{\mathrm{i} n \theta}-1\right)\right|^{2} \mathrm{d} \theta \mathrm{~d} A(z)\\
& \geq \lim _{|a| \rightarrow 1} \int_{\mathbb{D}}(1-|\psi_{a}(z)|^{2})\left|\sum_{n=1}^{\infty} n a_{n} z^{n-1} \int_{0}^{2 \pi}\left(e^{\mathrm{i} n \theta}-1\right) \mathrm{d} \theta\right|^{2} \mathrm{~d} A(z)\\
&\approx \lim _{|a| \rightarrow 1}\int_{\mathbb{D}}(1-|\psi_{a}(z)|^{2})\left|g^{\prime}(z)\right|^{2}\mathrm{~d} A(z),
\end{aligned}
$$
where the dominated convergence theorem and Fubini's theorem are applied to the second and third lines, respectively. Thus, by \eqref{c}, we obtain $g \in$ $\mathcal{V M O A}$.
\end{proof}

We can now prove our first main theorem.

\begin{proof}[Proof of Theorem \ref{thm-a}]
By (ii) of Theorem \ref{567}, $[C_\varphi,J_g]\in\mathcal{K}(\mathcal{H}^{p})$ for every $C_\varphi\in \mathscr B(\mathcal{H}^{p})$
	if and only if $g\in \mathcal{VMOA}$, which, according to Lemma \ref{181}, is equivalent to $J_g\in\mathcal{K}(\mathcal{H}^{p})$.
	Hence $D(J_g)$ maps bounded operators into $\mathcal{K}(\mathcal{H}^{p})$ if and only if $J_g$ is a compact operator.
\end{proof}

In the remaining part of this section, we characterize the compact intertwining relation for $J_{g}$ and $C_{\varphi}$ from $\mathcal{H}^{p}$  to $\mathcal{A}_{\alpha}^{q}$. For this purpose, we define the integral operator
$$
\mathrm{I\!I}_{\varphi, \alpha}^{p, q}(u)(a)=\int_{\mathbb{D}}\left(\frac{1-|a|^{2}}{|1-\bar{a}
	 \varphi(z)|^{2}}\right)^{\frac{q}{p}}|u(z)|^{q} (1-|z|^{2})^{\alpha} \mathrm{~d} A(z) .
$$

Using Remark \ref{Hp}, the following theorem can be
proved similarly to Theorem 1 and Corollary 1 in \cite{ZC}, and hence we omit the proof.
\begin{theorem}\label{w}
Let $1<p \leq q<\infty$ and $\alpha>-1.$ Assume that $u \in H(\mathbb{D})$  and $\varphi \in S(\mathbb{D})$.
\begin{itemize}
\item[(i)] The weighted composition operator $u C_{\varphi}$ is bounded  from $\mathcal{H}^{p}$ into $\mathcal{A}_{\alpha}^{q}$ if and only if
\begin{align}\label{200}
\sup _{a \in \mathbb{D}} \mathrm{I\!I}_{\varphi, \alpha}^{p, q}(u)(a)<\infty .
\end{align}
\item[(ii)] The weighted composition operator $u C_{\varphi}$ is compact from $\mathcal{H}^{p}$ into $\mathcal{A}_{\alpha}^{q}$ if and only if $u C_{\varphi}$ is bounded and
\begin{align}
\lim _{|a| \rightarrow 1} \mathrm{I\!I}_{\varphi,\alpha}^{p, q}(u)(a)=0 .
\end{align}
\end{itemize}
\end{theorem}

Using Corollary 4.3 in \cite{TZ1} and Theorem \ref{w}, the following result can be obtained immediately.
\begin{corollary}\label{10}
Let $1<p \leq q<\infty$ and  $\alpha>-1$. Assume that $\varphi \in$ $S(\mathbb{D})$ and  $g, h\in H(\mathbb D)$.
\begin{itemize}
\item[(i)] $T_{\varphi, g, h}$ is a bounded operator from $\mathcal{H}^{p}$ into $\mathcal{A}_{\alpha}^{q}$ if and only if
\begin{align}
\sup _{a \in \mathbb{D}} \mathrm{I\!I}_{\varphi, q+\alpha}^{p, q}\left((g \circ \varphi-h)^{\prime}\right)(a)<\infty .
\end{align}
\item[(ii)] $T_{\varphi, g, h}$ is a compact operator from $\mathcal{H}^{p}$ into $\mathcal{A}_{\alpha}^{q}$ if and only if $T_{\varphi, g}$ is bounded and
\begin{align}
\lim _{|a| \rightarrow 1} \mathrm{I\!I}_{\varphi, q+\alpha}^{p, q}\left((g \circ \varphi-h)^{\prime}\right)(a)=0 .
\end{align}
\end{itemize}
\end{corollary}

\begin{proposition}\label{234}
Let $1<p\leq q<\infty$, $g \in H(\mathbb{D}),-1<\alpha<\infty,$ and $\gamma=\frac{\alpha+2}{q}-\frac{1}{p}$. Then we have the following two assertions.
\begin{itemize}
\item[(i)] If $p<q$ and $\gamma+1 \geq 0$, then $J_{g}: \mathcal{H}^{p} \rightarrow \mathcal{A}_{\alpha}^{q}$ compactly
intertwines all composition operators $C_{\varphi}$ which are bounded both on $\mathcal{H}^{p}$
 and $\mathcal{A}_{\alpha}^{q}$ if and only if $g$ $\in$ $\mathcal{B}^{1+\gamma}_{0}$.
\item[(ii)] If $p=q$, then $J_{g}: \mathcal{H}^{p} \rightarrow \mathcal{A}_{\alpha}^{p}$ compactly
 intertwines all composition operators $C_{\varphi}$ which are bounded both on $\mathcal{H}^{p}$ and $\mathcal{A}_{\alpha}^{p}$
  if and only if $g$ $\in$ $\mathcal{V M O A}_{p}^{1+(\alpha+1) / p}.$
\end{itemize}
\end{proposition}
\begin{proof}
The proof is similar to that of Theorem \ref{567} and hence we omit it.
\end{proof}

\section{Proof of Theorem \ref{thm-b}}\label{proof-of-thm-b}

For $\varphi \in S(\mathbb{D})$ and $f, u \in H(\mathbb{D})$, a weighted differential 
composition operator is defined by
$$
u C_{\varphi}^{\prime} f:=u \cdot(f \circ \varphi)^{\prime}.
$$
For a similar role that the operator $T_{\varphi, g, h}$ played in the previous section, we define the operator
$$
\begin{aligned}
S_{\varphi, g, h}f(z)
&=\int_{0}^{\varphi(z)} f^{\prime}(w) g(w) \mathrm{d} w-\int_{0}^{z} (f \circ \varphi)^{\prime}(w) h(w) \mathrm{d} w.
\end{aligned}
$$
and the integral operator
$$
\mathrm{I\!I\!I}_{\varphi}^{p, q}(u)(a)=\int_{\mathbb{D}}\left(\frac{1-|a|^{2}}{|1-\bar{a} \varphi(z)|^{2}}\right)^{(3+\frac{2}{p}-\frac{2}{q})}|u(z)|^{2} (1-|\varphi(z)|^{2}) \mathrm{~d} A(z),
$$
which we use to characterize the properties of $S_{\varphi, g, h}$. Further, as in the previous section, we describe boundedness and compactness of $S_{\varphi, g, h}:\mathcal H^p\to \mathcal H^q$ in terms of the generalized area operators, which leads to a simple proof of Theorem~\ref{thm-b}. Finally, at the end of this section, we consider the compact intertwining relation for $I_{g}$ and $C_{\varphi}$ from $\mathcal{H}^{p}$  to $\mathcal{A}_{\alpha}^{q}$.

\medskip

We start with the following characterization, which is the main ingredient in the proof of Theorem \ref{thm-b}.

\begin{proposition}\label{77}
 Assume that $\varphi \in$ $Aut(\mathbb{D})$,  and $g,h\in H(\mathbb D)$. Let $0<p\leq q<\infty$.
We have the following two assertions:
\begin{itemize}
\item[(i)] If $S_{\varphi, g, h}$ is a bounded operator from $\mathcal{H}^{p}$ to $\mathcal{H}^{q}$ then
\begin{align}
\sup _{a \in \mathbb{D}} \mathrm{I\!I\!I}_{\varphi}^{p, q}\left[(g \circ \varphi-h)\varphi'\right](a)<\infty .
\end{align}
\item[(ii)] If $S_{\varphi, g, h}$ is a compact operator from $\mathcal{H}^{p}$ to $\mathcal{H}^{q}$ then 
\begin{align}
\lim _{|a| \rightarrow 1} \mathrm{I\!I\!I}_{\varphi}^{p, q}
\left[(g \circ \varphi-h)\varphi'\right](a)=0 .
\end{align}
\end{itemize}
Moreover, if $p, q$ satisfy one of the following
conditions:
\begin{itemize}
\item[(a)] $p\leq q=2$;
\item[(b)] $q>2>p^2q/(pq+q-p)$.
\end{itemize}
Then the two conditions above in \emph{(i)} and \emph{(ii)} are also sufficient.
\end{proposition}

\begin{proof}
We note that
$$
\begin{aligned}
\left(C_{\varphi} I_{g}-I_{h} C_{\varphi}\right) f(z) &=C_{\varphi}\left(\int_{0}^{z} f'(w) g(w) \mathrm{d} w\right)-I_{h}(f \circ \varphi)(z) \\
&=\int_{0}^{\varphi(z)} f'(w) g(w) \mathrm{d} w-\int_{0}^{z} (f\circ\varphi)'(w) h(w) \mathrm{d} w,
\end{aligned}
$$
thus, by Lemma \ref{sss}, we get
$$
\begin{aligned}
\|S_{\varphi, g, h} f\|_{\mathcal{H}^{q}}^{q}
& \approx \int_{\mathbb{T}} \left(\int_{\Gamma(\zeta)}|(g \circ \varphi-h) C_{\varphi}^{\prime}(f)(z)|^{2} \mathrm{d} A(z)\right)^{q/2} |\mathrm{d} \zeta|. \\
\end{aligned}
$$
Let $w=\varphi(z)$, $\mu_{\varphi}=\nu_{\varphi} \circ \varphi^{-1}$, and $\mathrm{d} \nu_{\varphi}(z)=\left|(g \circ \varphi-h)(z)\varphi'(z)\right|^{2}d A(z)$. Then
$$
\begin{aligned}
\|S_{\varphi, g, h} f\|_{\mathcal{H}^{q}}^{q}
& \approx \int_{\mathbb{T}} \left(\int_{\Gamma_{\gamma'}(\eta)}|f^{\prime}(w)|^{2} \frac{(1-|w|)\mathrm{d}\mu_{\varphi}(w)}{1-|w|}\right)^{q/2} |\mathrm{d} \eta| \\
&=\left\| \left(\int_{\Gamma_{\gamma'}(\eta)}|f^{\prime}(w)|^{2} \frac{\mathrm{d}\mu(w)}{1-|w|}\right)^{1/2}\right\|_{\mathcal{L}^{q}(\mathbb{T})}^{q}\\
& =\left\|\widetilde{A_{\mu}^{2}}(f)\right\|_{\mathcal{L}^{q}(\mathbb{T})}^{q},
\end{aligned}
$$
where $\mathrm{d}\mu(w)=(1-|w|)\mathrm{d}\mu_{\varphi}(w)$ and $\gamma'$ is chosen as in the proof of Proposition~\ref{yyy}.

Hence $S_{\varphi, g, h}: \mathcal{H}^{p} \rightarrow \mathcal{H}^{q}$ is bounded if and only if $\widetilde{A_{\mu}^{2}}: \mathcal{H}^{p} \rightarrow \mathcal{L}^{q}\left(\mathbb{T}\right)$ is bounded. Thus, Lemma \ref{66} (i) implies that $\mathrm{d} \mu$ is a $(3+\frac{2}{p}-\frac{2}{q})$-Carleson measure. By \eqref{13}, this is equivalent to
$$
\sup _{a \in \mathbb{D}} \int_{\mathbb{D}}\left(\frac{\left(1-|a|^{2}\right)}{|1-\bar{a} w|^{2}}\right)^{(3+\frac{2}{p}-\frac{2}{q})} (1-|w|)\mathrm{d} \mu_{\varphi}(w)<\infty.
$$
Changing the variable back to $z$ completes the proof of (i).

Similarly, we can prove (ii) and sufficiency from Lemma \ref{66}. We omit the details.
\end{proof}

We are now ready to prove Theorem \ref{thm-b}.

\begin{proof}[Proof of Theorem \ref{thm-b}]
Let $0<p<\infty$, and suppose that $D(I_g)$ maps $\mathscr B(\mathcal{H}^p)$ into the ideal of compact operators. Let $\varphi(z)=e^{\mathbf{i} \theta} z$ for $\theta \in$ $[0,2 \pi]$. Then 
$$
	S_{\varphi,g, g}=[C_\varphi,I_g]\in\mathcal{K}(\mathcal{H}^{p}).
$$
Therefore, by Proposition \ref{77} (ii),
\begin{align}\label{z1}
\begin{aligned}
0=&\lim _{|a| \rightarrow 1} \int_{\mathbb{D}}\left(\frac{1-|a|^{2}}{\left|1-\bar{a} e^{\mathbf{i} \theta} z\right|^{2}}\right)^{3}\left|g\left(e^{\mathbf{i} \theta} z\right)-g(z)\right|^{2}(1-|z|^{2}) \mathrm{~d} A(z) \\
\approx & \lim _{|a| \rightarrow 1}\left|g\left(e^{\mathbf{i} \theta} a\right)-g(a)\right|^2\left(1-|a|^{2}\right)^{0},
\end{aligned}
\end{align}
where the last line follows from Lemmas 8 and 9 of \cite{ZC}. Thus, by the uniqueness theorem, the function $g$ is a complex scalar.

Conversely, if $g$ is a complex scalar $c$, then $I_g= M_c$ and clearly
$$
	D(I_g)T = [M_c,T] = 0 \in \mathcal{K}(\mathcal{H}^{p})
$$
for every $T\in \mathscr B(\mathcal{H}^p)$, which completes the proof.

\end{proof}

In the remaining part of this section, we characterize the compact intertwining relation for
 $I_{g}$ and $C_{\varphi}$ from $\mathcal{H}^{p}$  to $\mathcal{A}_{\alpha}^{q}$. 
 We first give a result on the boundedness and compactness of $u C_{\varphi}^{\prime}$
  from $\mathcal{H}^{p}$ to $\mathcal{A}_{\alpha}^{q}$. To this end, we define another integral operator as follows
$$
\mathrm{I\!V}_{\varphi, \alpha}^{p, q}(u)(a):=\int_{\mathbb{D}}\left(\frac{1-|a|^{2}}{|1-\bar{a} \varphi(z)|^{2}}\right)^{(1+p)q/p}\left|u(z) \varphi^{\prime}(z)\right|^{q} (1-|z|^{2})^{\alpha}\mathrm{~d} A(z).
$$

\begin{theorem}\label{3c}
Suppose that either $2 \leq p=q$ or $0<p <q<\infty$, and let $\alpha>-1$, $u \in H(\mathbb{D})$  and $\varphi \in S(\mathbb{D})$.
\begin{itemize}
\item[(i)] The operator $u C_{\varphi}^{\prime}:\mathcal H^p\to \mathcal A^q_\alpha$ is bounded if and only if
\begin{align}
\sup _{a \in \mathbb{D}} \mathrm{I\!V}_{\varphi, \alpha}^{p, q}(u)(a)<\infty.
\end{align}
\item[(ii)] The operator $u C_{\varphi}^{\prime}:\mathcal H^p\to \mathcal A^q_\alpha$ is compact if and only if $u C_{\varphi}^{\prime}$ is bounded and
\begin{align}
\lim _{|a| \rightarrow 1} \mathrm{I\!V}_{\varphi,\alpha}^{p, q}(u)(a)=0 .
\end{align}
\end{itemize}
\end{theorem}

\begin{proof}
By definition, $u C_{\varphi}^{\prime}$ is bounded from $\mathcal{H}^{p}$ into $\mathcal{A}_{\alpha}^{q}$ if and only if
$$
\left\|\left(u C_{\varphi}^{\prime}\right) f\right\|_{\mathcal{A}_{\alpha}^{p}}^{q} \lesssim\|f\|_{\mathcal{H}^{p}}^{q},
$$  for all $f \in \mathcal{H}^{p}$.
Changing variables $w=\varphi(z)$ yields
\begin{align}\label{1a}
\int_{\mathbb{D}}|f^{\prime}(w)|^{q} \mathrm{~d} \mu_{u, \varphi}(w) \lesssim\|f\|_{\mathcal{H}^{p}}^{q},
\end{align}
where $\mu_{u, \varphi}=\nu_{u, \varphi} \circ \varphi^{-1}$ and $\mathrm{d} \nu_{u, \varphi}(z)
=\left|u(z) \varphi^{\prime}(z)\right|^{q} (1-|z|^{2})^{\alpha} \mathrm{~d} A(z)$. 
Hence by \eqref{13}, \eqref{1a} and Lemma \ref{12}, we have
$$
\sup _{a \in \mathbb{D}} \int_{\mathbb{D}}\left(\frac{1-|a|^{2}}{|1-\bar{a} w|^{2}}\right)^{(1+p)q/p} \mathrm{d} \mu_{u, \varphi}(w)<\infty.
$$
Changing the variable back to $z$ completes the proof of (i).

The proof (ii) is very similar to (i), and we just need Lemma \ref{2b}
to obtain the compactness of $u C_{\varphi}^{\prime}$ from $\mathcal{H}^{p}$ into $\mathcal{A}_{\alpha}^{q}$.
\end{proof}

By a direct calculation and Theorem \ref{3c}, the following corollary can be proved similarly to Corollary \ref{10}, and hence we omit the proof.

\begin{corollary}\label{4d}
Suppose that either $2 \leq p=q$ or $0<p <q<\infty$, and let $\alpha>-1$, $\varphi \in S(\mathbb{D})$, and $g, h\in H(\mathbb{D})$.
\begin{itemize}
\item[(i)]  $S_{\varphi, g, h}$ is a bounded operator from $\mathcal{H}^{p}$ into $\mathcal{A}_{\alpha}^{q}$ if and only if
\begin{align}
\sup _{a \in \mathbb{D}} \mathrm{I\!V}_{\varphi, q+\alpha}^{p, q}(g \circ \varphi-h)(a)<\infty .
\end{align}
\item[(ii)] $S_{\varphi, g, h}$ is a compact operator on $\mathcal{H}^{p}$ into $\mathcal{A}_{\alpha}^{q}$ if and only if $S_{\varphi, g}$ is bounded and
\begin{align}
\lim _{|a| \rightarrow 1} \mathrm{I\!V}_{\varphi, q+\alpha}^{p, q}(g \circ \varphi-h)(a)=0 .
\end{align}
\end{itemize}
\end{corollary}

\begin{proposition}\label{}
Suppose that either $2 \leq p=q$ or $0<p <q<\infty$, and let $g \in H(\mathbb{D}),-1<\alpha<\infty,$ and $\gamma=\frac{\alpha+2}{q}-\frac{1}{p}$. Then $I_{g}: \mathcal{H}^{p} \rightarrow \mathcal{A}_{\alpha}^{q}$ compactly intertwines all composition operators $C_{\varphi}$ which are bounded both on $\mathcal{H}^{p}$ and $\mathcal{A}_{\alpha}^{q}$ if and only if $g$ $\in$ $\mathcal{B}^{1+\gamma}_{0}$.
\end{proposition}
\begin{proof}
This is similar to the proof of Theorem \ref{567} and hence we omit the details.
\end{proof}

\section{Compact intertwining relation for $M_{g}$ and $C_{\varphi}$ from $\mathcal{H}^{p}$  to $\mathcal{A}_{\alpha}^{q}$}\label{section 5}

In this section, we discuss the compact intertwining relations for the
 multiplication operators $M_{g}$ and composition operators $C_{\varphi}$ from $\mathcal{H}^{p}$  to $\mathcal{A}_{\alpha}^{q}$.
To this end, we first characterize boundedness and compactness of
$$
N_{\varphi, g, h}:=C_{\varphi}M_{g}-M_{h}C_{\varphi},
$$
which is a linear operator from $\mathcal{H}^{p}$ to $\mathcal{A}_{\alpha}^{q}$.
\begin{corollary}\label{14}
Let $1<p \leq q<\infty$ and $\alpha>-1.$ Assume that $u \in H(\mathbb{D})$ and $\varphi \in S(\mathbb{D}).$
\begin{itemize}
\item[(i)] The multiplication operator $M_{u}$ is bounded from $\mathcal{H}^{p}$ into $\mathcal{A}_{\alpha}^{q}$ if and only if
\begin{align}
\sup _{a \in \mathbb{D}} \mathrm{I\!I}_{e^{\mathbf{i} \theta} z, \alpha}^{p, q}(u)(a) \approx \sup _{a \in \mathbb{D}}|u(a)|(1-|a|^{2})^{\gamma}<\infty .
\end{align}
\item[(ii)] The multiplication operator $M_{u}$ is compact from $\mathcal{H}^{p}$ into $\mathcal{A}_{\alpha}^{q}$ if and only if
\begin{align}
\limsup _{|a| \rightarrow 1} \mathrm{I\!I}_{e^{\mathbf{i} \theta} z, \alpha}^{p, q}(u)(a) \approx \limsup _{|a| \rightarrow 1}|u(a)|(1-|a|^{2})^{\gamma}=0,
\end{align}
where $\gamma=(\alpha+2)/q-1/p.$
\end{itemize}
\end{corollary}
\begin{proof}
By Lemma $8$ and Lemma $9$ in \cite{ZC}, we get 
$$\mathrm{I\!I}_{id, \alpha}^{p, q}(u)(a)=\mathrm{I\!I}_{e^{\mathbf{i} \theta} z, \alpha}^{p, q}(u)(a) 
\approx \sup _{a \in \mathbb{D}}|u(a)|(1-|a|^{2})^{\gamma}.$$ Note that
$M_{u}=uC_{id}$ where $id(z)\equiv z$ is the identity map of $\mathbb{D}$. Hence by setting $\varphi(z)=e^{\mathbf{i} \theta} z$ in Theorem \ref{w}, 
we can prove (i) and (ii).
\end{proof}

\begin{corollary}\label{15}
Let $1<p \leq q<\infty$ and $\alpha>-1$. Assume that $\varphi \in$ $S(\mathbb{D})$ and $g$, $h$ $\in$ $H(\mathbb{D})$.
\begin{itemize}
\item[(i)] The operator $N_{\varphi, g, h}$ is a bounded operator from $\mathcal{H}^{p}$ into $\mathcal{A}_{\alpha}^{q}$ if and only if
\begin{align}
\sup _{a \in \mathbb{D}} \mathrm{I\!I}_{\varphi, \alpha}^{p, q}\left(g \circ \varphi-h\right)(a)<\infty .
\end{align}
\item[(ii)] The operator $N_{\varphi, g, h}$ is a compact operator from $\mathcal{H}^{p}$ into $\mathcal{A}_{\alpha}^{q}$ if and only if $N_{\varphi, g, h}$ is bounded and
\begin{align}
\lim _{|a| \rightarrow 1} \mathrm{I\!I}_{\varphi, \alpha}^{p, q}\left(g \circ \varphi-h\right)(a)=0 .
\end{align}
\end{itemize}
\end{corollary}
\begin{proof}
We note that
$$
\left(C_{\varphi} M_{g}-M_{h} C_{\varphi}\right) f(z) = f(\varphi(z))(g(\varphi(z))-h(z)).
$$
Thus
$$
\begin{aligned}
\|N_{\varphi, g, h} f\|_{\mathcal{A}_{\alpha}^{q}}^{q}
&=  \int_{\mathbb{D}}\left|(g \circ \varphi-h)(z)\right|^{q}|f(\varphi(z))|^{q} \mathrm{~d} A_{\alpha}(z) \\
&=\left\|(g \circ \varphi-h) C_{\varphi}(f)\right\|_{\mathcal{A}_{\alpha}^{q}}^{q} .
\end{aligned}
$$
Hence, by Theorem \ref{w}, we have (i) and (ii).
\end{proof}

We can now state and prove the main result of this section.

\begin{theorem}\label{2345}
Let $1<p,q<\infty$, $g \in H(\mathbb{D}),-1<\alpha<\infty,$ and $\gamma=\frac{\alpha+2}{q}-\frac{1}{p}$. Then the following three assertions hold:
\begin{itemize}
\item[(i)]  If $p<q$ and $\gamma>0$, then $M_{g}: \mathcal{H}^{p} \rightarrow \mathcal{A}_{\alpha}^{q}$ 
compactly intertwines all composition operators $C_{\varphi}$ which are bounded both on $\mathcal{H}^{p}$
 and $\mathcal{A}_{\alpha}^{q}$ if and only if $g$ $\in$ $\mathcal{B}^{1+\gamma}_{0}$.
\item[(ii)] If $p<q$ and $\gamma=0$, then $M_{g}: \mathcal{H}^{p} \rightarrow \mathcal{A}_{\alpha}^{q}$ 
compactly intertwines all composition operators $C_{\varphi}$ which are bounded both on $\mathcal{H}^{p}$ 
and $\mathcal{A}_{\alpha}^{q}$ if and only if $g$ is a complex scalar.
    \item[(iii)] If $p=q=2$, then $M_{g}: \mathcal{H}^{2} \rightarrow \mathcal{A}_{\alpha}^{2}$ compactly
     intertwines all composition operators $C_{\varphi}$ which are bounded both on $\mathcal{H}^{2}$ and 
     $\mathcal{A}_{\alpha}^{2}$ if and only if $g$ $\in$$ \mathcal{V M O A}_{2}^{1+(\alpha+1) / 2}.$
\end{itemize}
\end{theorem}

\begin{proof}
The proof of (i) is similar to the proof of Theorem \ref{567}. Next we consider (ii). Sufficiency is obvious, and we just prove the necessity.
Putting $\varphi(z)=e^{\mathbf{i} \theta} z$ for $\theta \in$ $[0,2 \pi]$, by Corollary \ref{15} (ii), we obtain
\begin{align}\label{}
\begin{aligned}
0=&\lim _{|a| \rightarrow 1} \mathrm{I\!I}_{\varphi, \alpha}^{p, q}\left(g \circ \varphi-g\right)(a)\\
=&\lim _{|a| \rightarrow 1}\int_{\mathbb{D}}\left(\frac{1-|a|^{2}}{|1-\bar{a} e^{\mathbf{i} \theta}
	z|^{2}}\right)^{q/p}\left|g(e^{\mathbf{i} \theta} z)-g(z)\right|^{q} (1-|z|^{2})^{\alpha}\mathrm{~d} A(z)\\
=&\lim _{|a| \rightarrow 1}\mathrm{I\!I}_{id, \alpha}^{p, q}\left(g
\circ \varphi-g\right)(a)\\
\approx &\lim _{|a| \rightarrow 1}\left|g(e^{\mathbf{i} \theta} a)-g(a)\right|(1-|a|^{2})^{0}
\end{aligned}
\end{align}
where the last approximation follows from the proof of Corollary \ref{14}. Then the uniqueness
theorem and $0 \approx \lim _{|a| \rightarrow 1}\left|g\left(e^{\mathbf{i} \theta} a\right)
-g(a)\right|$ give that $g \equiv$ constant.

Finally, we prove (iii). If $g$ $\in$ $\mathcal{VMOA}_2^{1+(\alpha+1)/2}$, we can see $M_{g}$ is compact from $\mathcal{H}^{2}$ to $\mathcal{A}^2_\alpha$ by Lemma \ref{6f} (vi). Hence
$$
\left.C_{\varphi}\right|_{\mathcal{A}^{2}_\alpha} \left.M_{g}\right|_{\mathcal{H}^{2} \rightarrow \mathcal{A}^{2}_{\alpha}}-\left.M_{g}\right|_{\mathcal{H}^{2} \rightarrow \mathcal{A}^{2}_\alpha} \left.C_{\varphi}\right|_{\mathcal{H}^{2}}
$$
is compact for every $C_{\varphi}$ bounded on $\mathcal{H}^{p}$ and  $\mathcal{A}^2_\alpha$.

Conversely, by (iii) of Lemma \ref{6f}, $M_{g}$ is bounded from $\mathcal{H}^{2}$ to $\mathcal{A}^2_\alpha$
if and only if $g$ $\in$ $\mathcal{BMOA}_2^{1+(\alpha+1)/2}$. Putting $\varphi(z)=e^{\mathbf{i} \theta} z$ for $ \theta \in$ $[0,2 \pi]$, by Corollary \ref{15} (ii), we have 
\begin{align}\label{up}
\begin{aligned}
0=&\lim _{|a| \rightarrow 1} \mathrm{I\!I}_{\varphi, \alpha}^{2, 2}\left(g \circ \varphi-g\right)(a)\\
=&\lim _{|a| \rightarrow 1}\int_{\mathbb{D}}\frac{1-|a|^{2}}{|1-\bar{a} e^{\mathbf{i} \theta}
	z|^{2}}\left|g(e^{\mathbf{i} \theta} z)-g(z)\right|^{2} (1-|z|^{2})^{\alpha}\mathrm{~d} A(z)\\
 =&\lim _{|a| \rightarrow 1}\int_{\mathbb{D}}(1-|\psi_{a}(z)|^{2})\left|g(e^{\mathbf{i} \theta} z)-g(z)\right|^{2}(1-|z|^{2})^{\alpha-1}  \mathrm{~d} A(z)\\
=&\lim _{|a| \rightarrow 1} \int_{\mathbb{D}}(1-|\psi_{a}(z)|^{2})\left|e^{\mathbf{i} \theta} g^{\prime}\left(e^{\mathbf{i} \theta} z\right)-g^{\prime}(z)\right|^{2}(1-|z|^{2})^{\alpha+1} \mathrm{~d} A(z),
\end{aligned}
\end{align} where the last identity follows from Theorem 3.3 (2) of \cite{JR2} (using $n=0,1$, $p=2$, $q=\alpha+1$, and $s=1$). To use the dominated convergence theorem below, we first need to obtain an upper bound for \eqref{up} as follows
$$
\begin{aligned}
&\int_{\mathbb{D}}(1-|\psi_{a}(z)|^{2})\left|e^{\mathbf{i} \theta} g^{\prime}\left(e^{\mathbf{i} \theta} z\right)-g^{\prime}(z)\right|^{2} (1-|z|^{2})^{\alpha+1}\mathrm{~d} A(z) \\
 \leq &\int_{\mathbb{D}}\left(1-|\psi_{a}(e^{\mathbf{i} \theta }z)|^{2}\right)\left| g^{\prime}\left(e^{\mathbf{i} \theta} z\right)\right|^{2}(1-|z|^{2})^{\alpha+1} \mathrm{~d} A(z) \\ &+\int_{\mathbb{D}}\left(1-|\psi_{a}(z)|^{2}\right)\left|g^{\prime}(z)\right|^{2}(1-|z|^{2})^{\alpha+1} \mathrm{~d} A(z)  \\
\leq& 2\|g\|_{\mathcal{B M O A}_2^{1+(\alpha+1)/2}}<\infty,
\end{aligned}
$$
where the last line follows from \eqref{40}. Hence, \eqref{up} has an upper bound independent of $\theta$.

We also write $g(z)=\sum_{n=0}^{\infty} a_{n} z^{n}$, and integrate the right-hand side of \eqref{up} with respect to $\theta$ from 0 to $2 \pi$ as follows
$$
\begin{aligned}
0 &=\int_{0}^{2 \pi} \lim _{|a| \rightarrow 1}\int_{\mathbb{D}}(1-|\psi_{a}(z)|^{2})\left|e^{\mathbf{i} \theta} g^{\prime}\left(e^{\mathbf{i} \theta} z\right)-g^{\prime}(z)\right|^{2}(1-|z|^{2})^{\alpha+1} \mathrm{~d} A(z) \mathrm{d} \theta \\
&=\lim _{|a| \rightarrow 1} \int_{0}^{2 \pi}\int_{\mathbb{D}}(1-|\psi_{a}(z)|^{2})\left|e^{\mathbf{i} \theta} g^{\prime}\left(e^{\mathbf{i} \theta} z\right)-g^{\prime}(z)\right|^{2}(1-|z|^{2})^{\alpha+1} \mathrm{~d} A(z) \mathrm{d} \theta \\
&=\lim _{|a| \rightarrow 1} \int_{\mathbb{D}}(1-|\psi_{a}(z)|^{2})\int_{0}^{2 \pi}\left|\sum_{n=1}^{\infty} n a_{n} z^{n-1}\left(e^{\mathrm{i} n \theta}-1\right)\right|^{2} \mathrm{d} \theta (1-|z|^{2})^{\alpha+1}\mathrm{~d} A(z)\\
& \geq \lim _{|a| \rightarrow 1} \int_{\mathbb{D}}(1-|\psi_{a}(z)|^{2})\left|\sum_{n=1}^{\infty} n a_{n} z^{n-1} \int_{0}^{2 \pi}\left(e^{\mathrm{i} n \theta}-1\right) \mathrm{d} \theta\right|^{2} (1-|z|^{2})^{\alpha+1}\mathrm{~d} A(z)\\
&\approx \lim _{|a| \rightarrow 1}\int_{\mathbb{D}}(1-|\psi_{a}(z)|^{2})\left|g^{\prime}(z)\right|^{2}(1-|z|^{2})^{\alpha+1}\mathrm{~d} A(z),
\end{aligned}
$$
where the dominated convergence theorem and Fubini's theorem are applied to the second and third lines, respectively. Thus, by \eqref{c1}, we obtain $g \in$ $\mathcal{V M O A}_2^{1+(\alpha+1)/2}$.

\end{proof}

\begin{remark}
From our proof, we observe that composition operators, especially the rotation composition operators, play a crucial role in the study of inner derivations, and we expect that they will be useful for their further study.
\end{remark}

\section*{Acknowledgments}

H. Arroussi was supported by the European Union’s Horizon 2022 research and innovation programme under the Marie Skłodowska-Curie Grant Agreement No.\ 101109510. C. Tong and Z. Yuan were supported in part by the National Natural Science Foundations of China (Grant Nos.\ 12171136 and 12411530045). J. Virtanen was supported in part by Engineering and Physical Sciences Research Council (EPSRC) grants EP/X024555/1 and EP/Y008375/1.


\begin{thebibliography}{99}
\bibitem{PJ1}
P. Ahern and J. Bruna, Maximal and area integral characterizations of Hardy-Sobolev spaces in the unit ball of $\mathbb{C}^{n}$. Rev. Mat. Iberoamericana 4(1), (1988), 123-153.
 
\bibitem{AJ}
A. Aleman and J. A. Cima, An integral operator on $\mathcal{H}^{p}$ and Hardy’s inequality, J. Anal. Math. 85 (2001), 157-176.

\bibitem{A} H. Arroussi, Weighted composition operators on {B}ergman spaces {$A^p_\omega$}, Math. Nachr. 295 (2022), no. 4, 631--656.

\bibitem{AGV} H. Arroussi, H. Gissy, and J. A. Virtanen,
Generalized Volterra type integral operators on large Bergman spaces,
Bull. Sci. Math. 182 (2023), Paper No.~103226, 45 pp.

\bibitem{ASX}
R. Aulaskari, D. Stegenga, and J. Xiao, Some subclasses of $\mathcal{BMOA}$ and their characterization in terms of Carleson measures, Rocky Mountain J. Math. 26 (1996), 485-506.




\bibitem{BJ}
O. Blasco and H. Jarchow, A note on Carleson measures for Hardy spaces, Acta Sci. Math.
(Szeged) 71(2005), 371-389.

\bibitem{AC}
A. Calder\'{o}n, Commutators of singular integral operators, Proc. Nat. Acad. Sci. USA, 53 (1965), 1092-1099.

\bibitem{Ca} J. W. Calkin, Two-sided ideals and congruences in the ring of bounded operator in Hilbert space, Ann. Math., 42(2)(1941),
839-873.

\bibitem{LC1}
L. Carleson, An interpolation problem for bounded analytic functions, Amer. J. Math. 80 (1958), 921-930.

\bibitem{LC2}
L. Carleson, Interpolations by bounded analytic functions and the corona problem, Ann.  Math. 76 (1962), 547-559.

\bibitem{Cowen}  C. C. Cowen and B. D. MacCluer, Composition Operators on Spaces of Analytic Functions (CRC Press, Boca Raton, FL, 1995).

\bibitem{ZC}
 \v{Z}. \v{C}u\v{c}kovi\'{c} and R. Zhao, Weighted composition operators between different weighted Bergman spaces and different Hardy spaces, Illinois J. Math. 51(2) (2007), 479-498.

\bibitem{DP}
P. Duren, Theory of $H^{p}$ spaces. Academic Press, New York(2000).

\bibitem{GJ}
 J. Garnett, Bounded analytic functions, Academic Press, New York, (1981).

\bibitem{GLW}
 M. Gong, Z. Lou, Z. Wu, Area operators from $\mathcal{H}^{p}$ spaces to $\mathcal{L}^{q}$ spaces. Sci. China Math. 53(2), (2010), 357-366.

\bibitem{BH}
B. Korenblum, H. Hedenmalm, K. Zhu, Theory of Bergman spaces, Springer, 2000.


\bibitem{LLR1}
X. Liu, Z. Lou, R. Zhao, A new characterization of Carleson measures on the unit ball of $\mathbb{C}^{n}$, Integr. Equ. Oper. Theory 93, 51 (2021). https://doi.org/10.1007/s00020-021-02667-z.

\bibitem{Lue}
D. Luecking, Forward and reverse Carleson inequalities for functions in Bergman spaces and their derivatives. Amer. J. Math. 107 (1985), no. 1, 85--111.

\bibitem{PL}
X. Lv, J. Pau, Tent Carleson measures for Hardy spaces, J. Funct. Anal., 287, (2024), 110459.

\bibitem{SJA}
S. Miihkinen, J. Pau, A. Per\"{a}l\"{a} and M. F. Wang, Volterra type integration operators from Bergman spaces to Hardy spaces, J. Funct. Anal. 279(32), (2020), 108564.

\bibitem{PPP}
J. Pel\'{a}ez, Compact embedding derivatives of Hardy spaces into Lebesgue spaces. Proc. Amer. Math. Soc.,  144(4), (2015), 187-190.

\bibitem{PR}
J. Pel\'{a}ez, J. R\"{a}tty\"{a} and K. Sierra, Embedding Bergman spaces into tent spaces, Math. Z., 281(2015), 1215-1237.
  
\bibitem{PXJ}
R. Peng, X. L. Xing, L. Y. Jiang, Pointwise multiplication operators from Hardy spaces to weighted Bergman spaces in the unit ball of $\mathbb C^n$. Act. Math. Sci, 39(4) (2019), 1003-1016.

\bibitem{JR2}
J. R\"{a}tty\"{a}, {$n$}-th derivative characterisations, mean growth of derivatives and {$F(p,q,s)$}. Bull. Austral. Math. Soc., 68(3), (2003), 405--421.

\bibitem{JR}
J. R\"{a}tty\"{a}, Integration operator acting on Hardy and weighted Bergman spaces. Bull. Australian Math. Soc, 75(3), (2007), 431--446.

\bibitem{JS}
J. H. Shapiro, Composition operators and classical function theory, Spriger-Verlag, 1993.

\bibitem{ST}
V. S. Shulman and Y. V. Turovskii, Topological radicals and joint
spectral radius, Funktsional. Anal. i Prilozhen. 46(2012), 61-82;
translation in Funct. Anal. Appl., 46(2), (2012), 287-304.

\bibitem{SW}
I. M. Singer and J. Wermer, Derivations on commutative normed
algebras, Math. Ann., 129(1955), 260-264.

\bibitem{AR}
A. Siskakis and R. Zhao, A Volterra type operator on spaces of analytic functions. Function spaces (Edwardsville, IL, 1998). Contemp. Math., 232(2), (1999), 299-311.

\bibitem{ES}
E. Stein and R. Shakarchi, Complex analysis. Princeton university Press, Princeton and Oxford(2007).

\bibitem{TYZ}C. Tong, C. Yuan, Z. Zhou, Compact intertwining relations for composition
operators on $\mathcal{H}^\infty$ and the Bloch space, New York J. Math., 24, (2018), 611-629.

\bibitem{TZ1} C. Tong, Z. Zhou, Intertwining relations for Volterra operators on the Bergman space,
  Illinois J. Math., 57(1), (2013), 195-211.

\bibitem{TZ2} C. Tong, Z. Zhou, Compact intertwining relations for composition operators
between the weighted Bergman spaces and the weighted Bloch spaces, J. Korean Math.
Soc., 51(1), (2014), 125-135.

\bibitem{ZW}
Z. Wu, Area operator on Bergman spaces, Sci. China Ser. A 49 (2006), 987-1008.

\bibitem{Z1}
R. Zhao, On a general family of function spaces. Ann. Acad. Sci. Fenn. Math. Diss. 105 (1996).

\bibitem{Z2}
R. Zhao, On $\alpha$-Bloch functions and $\mathcal{V M O A}$. Acta Math. Sci. 16 (1996), 349-360.

\bibitem{ZRH}
R. Zhao, Pointwise multipliers from weighted Bergman spaces and Hardy spaces to weighted Bergman spaces.  Ann. Acad. Sci. Fenn. Math, 29(1) (2004), 139-150.

\bibitem{Zhu1}
 K. Zhu, Bloch type spaces of analytic functions, Rocky Mountain J. Math. 23(3) (1993), 1143-1177.
 
\bibitem{Zhu}
K. Zhu, Spaces of Holomorphic Functions in the Unit Ball, Grad. Texts in Math., Springer-Verlag, New York, 2005.

\end{thebibliography}
\end{document}